\documentclass[11pt, twoside]{amsart}
\calclayout

\title[Remarks on symmetric fusion categories of low rank]{Remarks on symmetric fusion categories of low rank in positive characteristic}

\address{}
\email{}
\urladdr{}
\usepackage{tikz-cd}
\usepackage{yfonts}
\usepackage{import}
\usepackage{amsmath}
\usepackage{amsfonts}
\usepackage{amsthm}
\usepackage{amssymb,bbm}
\usepackage{dsfont}
\usepackage{nicefrac}
\usepackage[english]{babel}
\usepackage{url}
\usepackage{fancyhdr}
\usepackage{graphicx}
\usepackage{verbatim}
\usepackage[inline]{enumitem}
\usepackage{mathtools}
\usepackage{hyperref}
\usepackage{microtype}
\usepackage{cleveref}

\newtheoremstyle{defstyle}
{0.5cm}                   
{0.5cm}                   
{\normalfont}           
{}     
{\normalfont\bfseries}  
{:}                     
{0.3cm}              
{\thmname{#1}\thmnumber{ #2}\thmnote{ (#3)}}

\numberwithin{equation}{section}
\DeclareRobustCommand{\SkipTocEntry}[5]{}

\newtheorem*{rep@theorem}{\rep@title}
\newcommand{\newreptheorem}[2]{%
	\newenvironment{rep#1}[1]{%
		\def\rep@title{#2 \ref{##1}}%
		\begin{rep@theorem}}%
		{\end{rep@theorem}}}
\makeatother

\newtheorem{theorem}{Theorem}[section]

\newtheorem{proposition}[theorem]{Proposition}
\newreptheorem{proposition}{Proposition}
\newtheorem{corollary}[theorem]{Corollary}
\newreptheorem{corollary}{Corollary}
\newtheorem{lemma}[theorem]{Lemma}

\newtheorem{theorem*}{Theorem}
\newreptheorem{theorem}{Theorem}

\theoremstyle{definition}

\newtheorem{example}[theorem]{Example}
\newtheorem{remark}[theorem]{Remark}

\newtheorem{question}[theorem]{Question}

\newcommand\pf{\begin{proof}}
\newcommand\epf{\end{proof}}

\newcommand\Ver{\operatorname{Ver}}
\newcommand\vect{\operatorname{Vec}}

\newcommand\End{\operatorname{End}}
\newcommand\Aut{\operatorname{Aut}}

\newcommand\Rep{\operatorname{Rep}}

\newcommand\FPdim{\operatorname{FPdim}}

\newcommand\Id{\operatorname{Id}}

\newcommand\coev{\operatorname{coev}}	
\newcommand\ev{\operatorname{ev}}

\newcommand\id{\operatorname{id}}

\newcommand\rank{\operatorname{rank}}

\makeatletter
\def\l@subsection{\@tocline{2}{0pt}{2pc}{5pc}{}}
\def\l@subsubsection{\@tocline{2}{0pt}{2.5pc}{5pc}{}}
\makeatother

\begin{document}

\author[A. Czenky]{Agustina Czenky}
\address{Department of Mathematics, University of Oregon, Eugene, OR 97403, USA}
\email{aczenky@uoregon.edu}

\begin{abstract}
	We give lower bounds for the rank of a symmetric fusion category in characteristic $p\geq 5$ in terms of $p$. We prove that the second Adams operation $\psi_2$ is not the identity for any non-trivial symmetric fusion category, and that symmetric fusion categories satisfying $\psi_2^a=\psi_2^{a-1}$ for some positive integer $a$ are super-Tannakian. As an application, we classify all symmetric fusion categories of rank 3 and those of rank 4 with exactly two self-dual simple objects.
\end{abstract}
\maketitle
\tableofcontents
\section{Introduction}
Fix an algebraically closed field \textbf{k} of characteristic $p \geq 0$. A \emph{symmetric fusion category} $\mathcal C$ over $\textbf{k}$ is a fusion category endowed with a braiding $c_{X,Y}: X\otimes Y \to Y \otimes X$ such that $c_{Y,X}c_{X,Y}=\id_{X\otimes Y} $ for all $X, Y \in  \mathcal{C}$,   
see \cite[Definition 8.1.12]{EGNO}. A well-known theorem by Deligne \cite{D} implies that a symmetric fusion category in characteristic $0$ is \emph{super-Tannakian}, that is, admits a symmetric tensor functor to the category $\operatorname{sVec}$ of super vector spaces. As a consequence of this theorem, a symmetric fusion category over a field of characteristic $0$ is equivalent to the category $\Rep_{\textbf k}(G ,z)$ of finite-dimensional representations of a finite group $G$. Here $z \in G$ is a central element of order $ 2$ that modifies the braiding, see  \cite[Section 8.19]{D}. This results gives a classification of symmetric fusion categories in characteristic zero in terms of group data.

Examples of symmetric fusion categories in positive characteristic are the Verlinde categories $\Ver_p$, defined as the semisimplification of the category of finite-dimensional $\textbf{k}$-representations of the cyclic group $\mathbb Z_p$ for $p$ a positive prime, see \cite[Section 3.2]{O2}. For $p\geq 5$, these categories have no fiber functor to $\operatorname{Vec}$ or $\operatorname{sVec}$, hence cannot be obtained as the category of representations of a finite group, see \cite{BEO}.

An important result by Victor Ostrik in \cite{O2} gives a new version of Deligne's theorem for the case of symmetric fusion categories in positive characteristic. He proved that any symmetric fusion category $\mathcal C$ in characteristic $p > 0$ admits a Verlinde fiber functor, that is, a $\textbf k$-linear exact symmetric tensor functor
 \begin{align*}
 	F:\mathcal C \to \Ver_p.
 \end{align*}
 As a consequence, any \textbf{k}-linear symmetric fusion category is equivalent to the category $\Rep_{\Ver_p}(G, \epsilon)$ of representations of some finite group scheme $G$ in $\Ver_p$ \cite[Corollary 1.6]{O2}. However, this statement does not give an explicit classification for $p\geq 5$, since  the classification of finite group schemes $G$ in $\Ver_p$ such that $\Rep_{\Ver_p}(G, \epsilon)$ is semisimple is not known, even when $\epsilon$ is trivial.

 When $p>0$, Nagata \cite[IV, 3.6]{DG} and Masuoka \cite{M} give a classification of finite group schemes $G$ in $\operatorname{Vec}$ and $\operatorname{sVec}$, respectively, such that $\Rep_{\textbf k}(G)$ is semisimple. This yields a reasonable classification of symmetric fusion categories in the super-Tannakian case. Note that when $\operatorname{char}(\mathbf{k})=2$ or $3$, symmetric fusion categories over \textbf{k} are Tannakian and super-Tannakian, respectively, so we know their classification.

 In this paper we will focus on the non super-Tannakian case. We will approach the classification of symmetric fusion categories in positive characteristic by rank, i.e., by the number of simple objects. Here is our first result.
 
 \begin{reptheorem}{rango p-1/2} Let $p\geq 5$. If $\mathcal C$ is a non super-Tannakian symmetric fusion category, then $$\rank(\mathcal C)\geq \frac{p-1}{2}.$$
 \end{reptheorem}
 
We note that the statement does not hold for super-Tannakian categories. For example, for $p\geq 3$ the category $\Rep(\mathbb Z_2)$ is semisimple and has rank $2$ which is strictly less than $\frac{p-1}{2}$ for $p>5$.

Note that equality in Theorem \ref{rango p-1/2} is achieved by $\Ver_p^+$, the fusion subcategory  of $\Ver_p$ generated by simple objects of odd index, see Section \ref{section: verlinde}. In characteristic $5$, it is known that the equality is only achieved by  $\Ver_5^+$, see \cite[ 4.6]{EOV}.

\begin{question}\label{question 1}
	Let $p>5$ and $\mathcal C$ a symmetric fusion category of rank $\frac{p-1}{2}$. If $\mathcal C$ is not super-Tannakian, is it true that $\mathcal C \cong \Ver_p^+$?
\end{question}

We give a positive answer for Question \ref{question 1} for the case $p=7$ in Theorem \ref{theo: rank 3}.

We also know that there exist non super-Tannakian symmetric fusion categories of rank $\frac{p+3}{2}$. In fact, let $\delta \in \textbf k$ and consider the Karoubian envelope $\underline{\operatorname{Rep}}(O(\delta))$ of the \emph{Brauer category} as defined in \cite[Section 9.3]{D}. Let $\underline{\operatorname{Rep}}^{\text ss }(O(\delta))$ denote the semisimplification of $\underline{\operatorname{Rep}}(O(\delta))$, i.e., the quotient of $\underline{\operatorname{Rep}}(O(\delta))$ by the tensor ideal of negligible morphisms, see e.g. \cite[Section 6.1]{D}. It turns out that when $\delta=-1$ this category contains a symmetric subcategory equivalent to $\Rep(\mathbb Z_2)$. The symmetric fusion category obtained by de-equivariantization by $\mathbb Z_2$ of the neutral component of the standard $\mathbb Z_2$-grading of  $\underline{\operatorname{Rep}}^{\text{ss}}(O(-1))$ has rank $\frac{p+3}{2}$, see  \cite{O3} for details.

We thus have examples of  non-super-Tannakian symmetric fusion categories in ranks $\frac{p-1}{2}$ and $\frac{p+3}{2}$. A natural question follows.
\begin{question}
	Are there non super-Tannakian symmetric fusion categories of rank $\frac{p+1}{2}$?
\end{question}

For $p\geq 5$, the category $\Ver_p$ has precisely four fusion subcategories: $\text{Vec}, \text{sVec}, \Ver_p^+$ and $\Ver_p$, see \cite[Proposition 3.3]{O2}. Thus, if $\mathcal C$ is not super-Tannakian, its Verlinde functor $F:\mathcal C \to \Ver_p$ is either surjective or its image is $\Ver_p^+$.  Our next result gives an improvement on the bound for the former case. 
 
 \begin{reptheorem}{theo: p-1} Let $p\geq 5$ and
 	let $\mathcal C$ be a symmetric fusion category with Verlinde fiber functor $F\colon \mathcal C \to \Ver_p$.  If $F$ is surjective then $$\rank(\mathcal C)\geq p-1.$$
 \end{reptheorem}

The main tool in the proof of Theorems \ref{rango p-1/2} and \ref{theo: p-1} is Galois theory. 

Another useful tool for the classification of symmetric fusion categories in positive characteristic is the \emph{second Adams operation}. Let $p\ne 2$. For a symmetric fusion category $\mathcal C $ with Grothendieck ring $\mathcal{K(C)}$, the second Adams operation is the ring endomorphism  $\psi_2:\mathcal{K(C)}\to \mathcal{K(C)}$ given by  $$\psi_2(X)=S^2(X)-\Lambda^2(X),$$ for all $X$ in $\mathcal C$, see \cite{EOV}.

	\begin{reptheorem}{remark on powers of psi}
	Let $p>2$ and let $\mathcal C$ be a non-super-Tannakian symmetric fusion category. If  the Adams operation $\psi_2: \mathcal{K(C)}\to \mathcal{K(C)}$ satisfies $\psi_2^a=\psi_2^b$ for some $a, b\in \mathbb Z_{\geq 0}$, then $2^a\equiv \pm 2^b \mod p$.
\end{reptheorem}
 
 \begin{repcorollary}{tannakian}
 Let $p> 2$ and let $\mathcal C$ be a symmetric fusion category. If $\psi_2^a=\psi_2^{a-1}$ for some $a\geq 1$, then $\mathcal C$ is super-Tannakian.
 \end{repcorollary} 

The following comes as a consequence.
 
 \begin{reptheorem}{psi not identity}
 	Let $p\ne 2$. If $\mathcal C$ is a non-trivial symmetric fusion category  then $\psi_2$ is not the identity. 
 \end{reptheorem}

We apply the second Adams operation to the problem of classification of symmetric fusion categories of low rank in positive characteristic. In \cite{EOV}, the second Adams operation was employed to give a complete clasification for rank 2. We classify symmetric fusion categories of rank $3$, and symmetric fusion categories of rank $4$ with exactly two self-dual simple objects, see Theorems \ref{theo: rank 3} and \ref{rank 4}, respectively. We also note that by Theorem \ref{rango p-1/2} non super-Tannakian symmetric fusion categories of rank $4$ are only possible in characteristic $p=5$ or $7$. 
 
Even though our results show that the second Adams operation is non-trivial for non-trivial symmetric fusion categories, we note that it is useful for the classification problem but definitely not sufficient on its own, see Remark \ref{Remark on Adams operation}.

 \medbreak
 This paper is organized as follows. A brief introduction to Verlinde categories and the second Adams operation is given in Section \ref{section:preliminaries}. Proofs for Theorems \ref{rango p-1/2} and \ref{theo: p-1} are given in Section \ref{section: bounds}. In Section \ref{section: Adams operation} we study properties of the second Adams operation, give a formula for it in the Verlinde category, and prove that it is not the identity for non-trivial fusion categories. Sections \ref{section: rank 3}  and \ref{section: rank 4} describe the classification of fusion categories of rank 3 and those of rank 4 with exactly two self-dual simple objects, respectively. We also say some words about the classification of those of rank 4 where all simple objects are self-dual in Section \ref{section: rank 4}. 
 
 \addtocontents{toc}{\SkipTocEntry}
 \subsection*{Acknowledgments}
 I am deeply grateful to my advisor Victor Ostrik for suggesting this project, providing insightful advice in how to approach the proofs of the main results, and for his numerous comments that made the exposition of this work much clearer. I also thank the referees for carefully reading this work and for their many helpful suggestions.

\section{Preliminaries}\label{section:preliminaries}
Throughout this paper \textbf{k} will denote an algebraically closed field of characteristic $p\geq 0$.

For a ring $R$, we denote by $R_{\mathbb Q}$ the scalar extension $R \otimes_{\mathbb Z} \mathbb Q$. If $z$ is a complex number, we denote by $\mathbb Q(z)$ the field extension generated by $z$ over $\mathbb Q$, and by $[\mathbb Q(z):\mathbb Q]$ the degree of said extension.

\subsection{Fusion categories}
In this section, we recall some useful definitions regarding fusion categories.

A \emph{tensor category} $\mathcal C$ is a $\textbf{k}$-linear abelian rigid monoidal category, with finite-dimensional Hom-spaces, and such that $\End_{\mathcal C}(\textbf 1)\cong \textbf k$, see e.g. \cite{SR} or \cite{EGNO}. We denote its tensor product  functor by $\otimes : \mathcal C \times \mathcal C\to \mathcal C$.


A \emph{fusion category} $\mathcal C$  is a tensor category which is semisimple with a finite number of isomorphism classes of simple objects. A \emph{fusion subcategory} of a fusion category $\mathcal C$ is a full tensor subcategory $\mathcal C' \subset \mathcal C$ such that if $X \in \mathcal C$ is isomorphic to a direct summand of an object of $\mathcal C'$, then $X \in \mathcal C'$, see \cite[2.1]{DGNO}. 

For fusion categories $\mathcal C$ and $\mathcal D$, a \emph{tensor functor} $F:\mathcal C\to \mathcal D$ is a $\textbf{k}$-linear exact and faithful monoidal functor, see \cite[Definition 4.2.5]{EGNO}. For a tensor functor $F:\mathcal C \to \mathcal D$, its image $ F(\mathcal C)$ is the fusion subcategory of $\mathcal D$ generated by objects $F(X)$, $X \in \mathcal C$. The functor $F$ is called \emph{injective} if it is fully faithful, and \emph{surjective} if $F(\mathcal C)=\mathcal D$, see \cite[5.7]{ENO}. Thus a tensor functor is an equivalence if and only if it is both injective and surjective.

For two fusion categories $\mathcal C$ and $\mathcal D$, we can define their \emph{external tensor product}, see \cite[Section 2.2]{O2}, which we will denote by $\mathcal C\boxtimes \mathcal D$.

\subsubsection{Frobenius-Perron dimension} Let $\mathcal C$ be a fusion category. We will denote by $\mathcal{K(C)}$ its Grothendieck ring, see e.g. \cite[4.5]{EGNO}. For an object $X$ in $\mathcal C$ we will use the same notation for its class $X$  in $\mathcal{K(C)}$. We recall that there is a unique ring homomorphism $\FPdim : \mathcal{K(C)} \to \mathbb R$ called \emph{Frobenius-Perron dimension} such that $\FPdim(X) \geq 1$ for
any object $X \ne 0$, see \cite[Proposition 3.3.4]{EGNO}. The \emph{Frobenius-Perron dimension} $\FPdim(\mathcal C)$ of $\mathcal C$ is defined as
\begin{align*}
	\FPdim(\mathcal C) = \sum\limits_{X } \FPdim(X)^2,
\end{align*}
where $X$ runs over a set of representatives of isomorphism classes of simple objects. We say that $\mathcal C$ is \emph{weakly integral} if $\FPdim(\mathcal C)$ is an integer, and \emph{integral} if $\FPdim(X)$ is an integer for all simple objects $X$. 

\subsection{Symmetric fusion categories}

Let $\mathcal{C}$ be a braided fusion category and denote by $c_{X,Y}$ the braiding morphism $X\otimes Y \to Y \otimes X$. We say $\mathcal C$ is \emph{symmetric} if 
\begin{align*}
	c_{Y,X}c_{X,Y}=\id_{X\otimes Y} \ \ \text{for all}\ X \in \ \mathcal{C}, 
\end{align*}  
see \cite{ EGNO}.
A \emph{symmetric tensor functor} between symmetric tensor categories is a tensor functor compatible with the commutativity isomorphism. 

We denote by Vec (respectively sVec) the symmetric fusion category of finite-dimensional vector spaces (respectively super vector spaces) over $\textbf{k}$. 

We say a symmetric fusion category $\mathcal C$ is \emph{Tannakian} (resp., \emph{super-Tannakian}) if it admits a symmetric fiber functor, that is, a symmetric tensor functor $\mathcal C \to \operatorname{Vec}$ (resp., $\mathcal C\to \operatorname{sVec}$), see \cite{SR, DM, D}. 

\subsubsection{The second Adams operation}\label{subsection:Adams operation}
Let $\mathcal C$ be a symmetric fusion category over a field of characteristic $p\ne 2$. 

We recall the definition of the second symmetric and exterior powers of an object $X\in \mathcal C$, following \cite[Definition 9.9.5]{EGNO} and \cite[2.1]{EHO}. Consider the action of the braid group $B_2$ on $X^{\otimes 2}$ in $\mathcal C$, see \cite[Remark 8.2.5]{EGNO}. This action factors through the symmetric group $S_2$. The \emph{second symmetric power} $S^2(X )$ of $X$ is the maximal quotient of $X^{\otimes 2}$ on which the action of $S_2$ is trivial. The \emph{second exterior power} $\Lambda^2(X)$ of $X$ is the maximal quotient of $X^{\otimes 2}$ on which the action of $S_2$ factors through the sign representation.

The \emph{second Adams operation} $\psi_2:\mathcal{K(C)}\to \mathcal{K(C)}$ is defined by 
\begin{align*}
	\psi_2(X)=S^2(X)-\Lambda^2(X),
\end{align*}
for all $X\in \mathcal{K(C)}.$  This defines a ring endomorphism of $\mathcal{K(C)}$, see \cite[Lemma 4.4]{EOV}.

Since $X^2=S^2(X)+\Lambda^2(X)$ for all $X\in \mathcal C$, then 
\begin{align*}
	X^2\equiv \psi_2(X) \mod 2 \ \ \text{for all } X\in \mathcal{K(C)}.
\end{align*}
We will use this fact repeatedly throughout this work. We also have that $\psi_2$ commutes with duality, that is, $\psi_2(X)^*=\psi_2(X^*)$ for all $X\in \mathcal{K(C)}$. 

When studying properties of the second Adams operation, we will often  look at its scalar extension $$(\psi_2)_{\mathbb Q} :=\psi_2\otimes 1: \mathcal{K(C)}_{\mathbb Q} \to \mathcal{K(C)}_{\mathbb Q}.$$

\subsection{Non-degenerate fusion categories}\label{section:non-deg}
Let $\mathcal C$ be a fusion category. Recall that a \emph{pivotal structure} on $\mathcal C$ is a tensor isomorphism $X\cong X^{**}$ for any $X\in \mathcal C$, see \cite{BW, EGNO}. Associated to a pivotal structure we can define the left and right trace of a morphism $X\to X$, see e.g. \cite[4.7]{EGNO}. The pivotal structure is called \emph{spherical} if for any such morphism its right trace equals its left trace. A \emph{spherical fusion category} is a fusion category equipped with a spherical structure. In the case when $\mathcal C$ is symmetric, there is a canonical choice of spherical structure given by 
	\begin{equation*}
	\begin{tikzcd}[column sep=huge,row sep=huge]
		X  \arrow{r}{\Id_X\otimes \coev_{X^*}} &X\otimes X^*\otimes X^{**} \arrow{r}{c_{X,X^*}\otimes \id_{X^{**}}} 
		&X^{*}\otimes X \otimes X^{**} \arrow{r}{\ev_{X}\otimes \Id_{X^{**}}} &X^{**},
	\end{tikzcd}
\end{equation*}
see e.g. \cite[Secion 9.9]{EGNO}.

Let $\mathcal C$ be a spherical fusion category. We recall the definition of \emph{dimension} $\dim(X)\in \textbf{k}$ of an object $X$ as the trace of its identity morphism. This determines a ring homomorphism $\dim: \mathcal{K(C)} \to \textbf{k}$ sending $X$ to $\dim(X)$. By \cite[Proposition 4.8.4]{EGNO} if $X\in\mathcal O(C) $ then $\dim(X)\ne 0$. 

The \emph{global dimension} $\dim(\mathcal C)\in \textbf{k}$ of a spherical fusion category $\mathcal C$ is defined as
\begin{align*}
	\dim(\mathcal C)=\sum\limits_{X\in \mathcal O(C)} \dim(X)^2.
\end{align*}
We say $\mathcal C$ is \emph{non-degenerate} if $\dim(\mathcal C)\ne 0$, see \cite[Definition 9.1]{ENO}. A crucial property of non-degenerate fusion categories is that they can be lifted to characteristic zero, see \cite{E} and \cite[Section 9]{ENO}. It is known that for $p=0$ any fusion category is non-degenerate, see \cite[Theorem 2.3]{ENO}.

\subsection{Verlinde categories}\label{section: verlinde} Let $p>0$. Let $\mathbb Z_p$ be the cyclic group of $p$ elements with generator $\sigma$. We have an isomorphism of algebras $\textbf{k}[\mathbb Z_p] = \textbf{k}[\sigma]/(\sigma^p-1)= \textbf{k}[\sigma]/(\sigma - 1)^p$. Thus isomorphism classes of indecomposable objects in the category $\Rep_{\textbf k}(\mathbb Z_p)$ are given by the $\mathbb Z_p$-modules $\tilde L_s:=\textbf{k}[\sigma]/(1 - \sigma)^s$,  for $s \in \mathbb Z$ satisfying $1 \leq s \leq p$.

The Verlinde category $\Ver_p$ is the symmetric fusion category over \textbf{k} obtained by quotienting $\Rep_{\textbf k}(\mathbb Z_p)$  by the tensor ideal of negligible morphisms, see \cite{O2} for details. Simple objects of this category are precisely the images of the indecomposable objects $\tilde L_s$ for $s=1, \dots, p-1$. We denote them by $\textbf{1}=L_1, L_2, \dots, L_{p-1}$. The Verlinde fusion rules are given by 
\begin{align*}
	L_r\otimes L_s = \sum\limits_{i=1}^{\min(r,s,p-r,p-s)} L_{|r-s|+2i-1}.
\end{align*}

We denote by $\Ver_p^+$  the abelian subcategory of $\Ver_p$ generated by $L_i$ for $i$ odd. By the Verlinde fusion rules, it turns out that $\Ver_p^+$ is a fusion subcategory of $\Ver_p$. For $p>2$ the fusion subcategory generated by $L_1$ and $L_{p-1}$ is tensor equivalent to $\text{sVec}$.  We have an equivalence of categories
\begin{align}
	\Ver_p \cong \Ver^+_p\boxtimes \  \text{sVec}, 
\end{align}
see \cite{O2}.

\subsubsection{The Verlinde fiber functor} Let $\mathcal C$ be a symmetric fusion
category over $\textbf k$ of characteristic $p>0$. The  main result of \cite{O2} states:
 \begin{theorem*}
 	\cite[Theorem 1.5]{O2} There exists a symmetric tensor functor $F :\mathcal  C \to \Ver_p$.
 \end{theorem*}
The functor $F$ is called the \emph{Verlinde fiber functor}.  It is shown in \cite[Theorem 2.6]{EOV} that it is unique up
to a non-unique isomorphism of tensor functors.

\section{Bounds for the ranks of  symmetric fusion categories}\label{section: bounds}
 In this section we prove our two main results concerning the ranks of non-super-Tannakian symmetric fusion categories. Throughout this section, we assume $\operatorname{char}(\textbf k)=p\geq 5$. Let $z=e^{2\pi i/p}$ be a primitive $p$-th root of unity. Recall that we denote by $\mathbb Q(z)$ the field extension generated by $z$ over $\mathbb Q$.
 
 	Let $\mathcal C$ be a symmetric fusion category and consider its Verlinde fiber functor $F:\mathcal C\to \Ver_p$. Recall that by \cite[Proposition 3.3]{O2} we have an equivalence of symmetric fusion categories
 \begin{align*}
 	\Ver_p \cong \Ver_p^+\boxtimes \text{ sVec}.
 \end{align*}
Consider the monoidal (non symmetric)  forgetful functor $\operatorname{Forget}:\operatorname{sVec} \to \operatorname{Vec}$. We have a (possibly non symmetric) tensor functor 
\begin{align}\label{tilde F}
	\tilde F:= (\operatorname{id}\boxtimes \operatorname{Forget}) \circ F: \mathcal C\to \Ver_p^+;
\end{align}
we denote also by $\tilde F$ the induced ring homomorphism $\mathcal{K(C)}\to 	\mathcal K(\Ver_p^+)$, and the induced $\mathbb Q$-algebra homomorphism $\mathcal{K(C)}_{\mathbb Q}\to 	\mathcal K(\Ver_p^+)_{\mathbb Q}$.  We are interested in studying the image of this map. By \cite[Theorem 4.5 (iv)]{BEO}, we have an isomorphism of $\mathbb Q$-algebras
 \begin{align*}
 	\mathcal K(\Ver_p^+)_{\mathbb Q}\cong \mathbb Q(z+z^{-1}),
 \end{align*}
and so $\tilde F(\mathcal{K(C)}_{\mathbb Q})$ is a subalgebra of $\mathbb Q(z+z^{-1})$. Since a subalgebra of a finite field extension is a field, then $\tilde F(\mathcal{K(C)}_{\mathbb Q})$ is a subfield of  $\mathbb Q(z+z^{-1})$. Hence to study the image of $\mathcal{K(C)}_{\mathbb Q}$ under $\tilde F$, we start by looking at subfields of $\mathbb Q(z+z^{-1})$.

 We make the usual identification of the Galois group of $\mathbb Q(z)$ with the multiplicative group $\mathbb Z_p^{\times}$. This is a cyclic group with $p-1$ elements, where $j$ acts on $\mathbb Q(z)$ by $j\cdot z= z^j$ for all $ j\in \mathbb Z_p^{\times}$. We denote the Galois group of the maximal real subextension $\mathbb Q(z+z^{-1}) $ of $\mathbb Q(z)$  by $\mathcal G$, which corresponds to the quotient of $\mathbb Z_p^{\times}$ by the subgroup $\{\pm 1\}$. Thus $\mathcal G$ is a cyclic group of order  $\frac{p-1}{2}$.
 
 By Galois correspondence, subextensions of $\mathbb Q(z+z^{-1})$ are in bijection with subgroups of $\mathcal G$. That is, for every positive integer $k$ that divides $\frac{p-1}{2}$ there exists a unique subextension $A_k$ of $\mathbb Q(z+z^{-1})$  such that $[A_k:\mathbb Q]=k$, and its Galois group is exactly the quotient of $\mathcal G$ by the unique subgroup $H_m$ of order $m$, where $mk=\frac{p-1}{2}$. Moreover, every subextension is of this form, and $A_k$ is the set of elements fixed by every element in $H_m$.

Consider the basis $\left\{z^i+z^{-i}\right\}_{i=1}^{\frac{p-1}{2}}$ of $\mathbb Q(z+z^{-1})$. Then the group $\mathcal G$ (and thus also all subgroups $H_m$)  acts on this set freely and transitively by permutation, $$ a \cdot (z^j+z^{-j}) = z^{aj}+ z^{-aj},$$ 
for all $a\in \mathcal G$. So the orbits of the action of $H_m$ on this set have exactly $m$ elements. Let $\mathcal O_1, \dots, \mathcal O_k$ denote the orbits of the action of $H_m$, and define
\begin{align}\label{x_i basis}
	x_i := \sum\limits_{z^t +z^{-t}\in \mathcal O_i} (z^t+z^{-t}),
\end{align}
so that $\left\{x_1, \dots, x_{k}\right\}$ is a basis of $A_k$.  Without loss of generality, we choose the labelling so that $z+z^{-1}=z^{p-1}+z^{-(p-1)}\in \mathcal O_k$,

\begin{theorem}\label{rango p-1/2}
	Let $p\geq 5$. If $\mathcal C$ is a non-super-Tannakian symmetric fusion category, then $$\rank(\mathcal C)\geq \frac{p-1}{2}.$$
\end{theorem}

\begin{proof}
	
	Consider the tensor functor $\tilde F:\mathcal C\to \Ver_p^+$  as defined in Equation \eqref{tilde F}. According to \cite[Theorem 4.5 (iv)]{BEO}, under the isomorphism
 \begin{align*}
 	\mathcal K(\Ver_p^+)_{\mathbb Q}\cong \mathbb Q(z+z^{-1}),
 \end{align*}
we have identifications
 \begin{align}\label{identification in objects}
 	L_{2j+1}=\sum\limits_{l=1}^j (z^{2l}+z^{-2l})+1, \ \ \text{for } j=0, \dots, (p-3)/2.
 \end{align}

Recall  that $\tilde F(\mathcal{K(C)}_{\mathbb Q})$ is a subfield of $\mathcal{K}(\Ver_p^+)_{\mathbb Q}\cong \mathbb Q(z+z^{-1})$. Now, 
\begin{align*}
	\rank(\mathcal{K(C)})\geq \rank(\tilde F(\mathcal{K(C)}) ) = \dim(\tilde F(\mathcal{K(C)}_{\mathbb Q})),
\end{align*}
where $\rank(\mathcal{K(C)})$ refers to the rank of $\mathcal{K(C)}$ as a free abelian group. 
Thus we would like to show that $\dim(\tilde F(\mathcal{K(C)}_{\mathbb Q}))=\frac{p-1}{2}$, or in other words, that $\tilde F(\mathcal{K(C)}_{\mathbb Q})=\mathbb Q(z+z^{-1})$.

By our discussion at the beginning of this section, we know that  $\tilde F(\mathcal{K(C)}_{\mathbb Q})$ is of the form $A_k$ for some $k$ that divides $\frac{p-1}{2}$, where $A_k$ is the unique subextension of order $k$. 
 Hence to prove the statement it is enough to show that $k=\frac{p-1}{2}$. Note that $k>1$, since if $k=1$ the image of $\tilde F$ would be a multiple of the identity and we would have a symmetric tensor functor from $\mathcal C$ to $\operatorname{sVec}$, which would mean that $\mathcal C$ is super-Tannakian. 

For objects $X\in \mathcal C$, their images under $\tilde F$ are objects in $\Ver_p^+$, and thus can be written as $\mathbb Z_{\geq 0}$ linear combinations of $L_i$'s with $i$ odd. Then $\tilde F(\mathcal {K(C)})$ has a basis of elements of this form, and thus so does $A_k$. That is, 
\begin{align}\label{A_k}
	\text{$A_k$ has a basis given by  $\mathbb Z_{\geq 0}$ linear combinations of  $L_i$'s for $i$ odd.}
\end{align}

For the sake of contradiction, assume that $k<\frac{p-1}{2}$. 
We already know that $k>1$.

Using formula \eqref{identification in objects} we compute
\begin{align*}
	&	z^t+z^{-t}\ =\  L_{t+1}-L_{t-1} \ \ \text{for } t \text{ even,} \  2\leq t<p-1. 
\end{align*}		
On the other hand, 
\begin{align*}
		& L_{p-2}= \sum_{l=1}^{\frac{p-3}{2}} (z^{2l}+z^{-2l}) +1 = - (z^{p-1}+z^{-(p-1)}),
\end{align*}
since $\sum\limits_{\substack{i=-(p-1)\\ i \text{ even}}}^{p-1} z^i=0.$
Let $\mathcal O_1, \dots, \mathcal O_k$ and $x_1, \dots, x_k$ be as in \eqref{x_i basis}. Since we can pick $t$ to be a positive even number for each summand $z^t+z^{-t}$ of $x_i$ (if not, replace $t$ by $-t$ or  $t-p$), then we can identify each $x_i$ with a sum of $L_s$'s with multiplicity $\pm 1$, as follows:
\begin{align*}
	x_i &= \sum\limits_{\substack{z^t +z^{-t}\in \mathcal O_i\\ t \text{ even}\\ 2\leq t <p-1}} (L_{t+1}-L_{t-1}), \text{ for } i \ne k,&\text{and}
&&	x_k = -L_{p-2}+ \sum\limits_{\substack{z^t +z^{-t}\in \mathcal O_k\\ t \text{ even}\\ 2\leq t <p-1}} (L_{t+1}-L_{t-1}) .
\end{align*}
Let $s$ odd, $1<s\leq  p-2$. Note that $L_s$ appears with nonzero multiplicity in either two basis elements, with multiplicity $1$ and $-1$, respectively, or in none (since it may cancel out with itself). On the other hand, $L_{1}$ appears in only one basis element (explicitly, the basis element $x_j$ such that  $z^2+z^{-2}\in \mathcal O_j$), with multiplicity $-1$. We will say $L_s$ is a  ``positive" summand of $x_i$ if it has multiplicity 1 in $x_i$, and is a ``negative" summand if it has multiplicity $-1$.

We claim that every $x_i$ has at least one positive and one negative summand. In fact, this is clear for $1\leq i<k$, since the number of positive summands in $x_i$ is the same as the number of negative summands. Suppose for contradiction that we have $x_k=-L_s$ for some even $s$, $2\leq s\leq p-2$. Our assumption $k<\frac{p-1}{2}$ assures  that every orbit has at least two elements, so it is not possible to have $x_k=-L_{p-2}$. Since we are assuming $x_k$ has only one negative summand, $L_{p-2}$ must cancel out with a positive summand. This implies $z^3+z^{-3}=z^{p-3}+z^{-(p-3)}\in \mathcal O_k$ as well, and so $H_m$, the unique subgroup of order $m$, contains the class $\bar 3$ of the number $3\in \mathbb Z_p^{\times}$, see discussion around Equation \eqref{x_i basis}. Now, either $x_k=-L_{p-4}$, or $-L_{p-4}$ cancels out. In the latter case, we have that $z^5+z^{-5}=z^{p-5}+z^{-(p-5)}\in \mathcal O_k$ and so $\bar 5 \in H_m$. Recursively, we get that $H_m=\{\bar 1, \bar 3, \bar 5, \dots, \bar j\}$ for some odd $3\leq  j\leq p-2$. We claim this contradicts that $H_m$ is a proper subgroup. In fact, since $H_m$ is a subgroup, it must contain the classes of $3l$ for all $l$ odd, $1\leq l \leq j$. Let $l\in H_m$ such that $3l\leq j < 3(l+2)$. Note that $l+2$ is also in $H_m$ (if not, then $3l\leq j <l+2$, and so $l<1$, which is not possible). So $3(l+2)$ must also be in $H_m$. But since $j<3(l+2)$, it must be the case that $3(l+2) > p$ and $\overline{ 3(l+2)}=\bar n$ for some odd $1\leq n<j$. So we have the inequalities
\begin{align*}
	3l\leq j <p<3(l+2),
\end{align*}
which imply $p=3l+2$ or $p=3l+4$. If $p=3l+2$, since $H_m$ contains the classes of all odd elements from $1$ to $3l=p-2$ we get that $|H_m|=\frac{p-1}{2}$, a contradiction. If $p=3l+4$, then $H_m$ contains all odd elements from $1$ to $3l=p-4$ (its missing at most one element) and thus again $H_m$ must have all odd elements, a contradiction. Hence $x_k$ has at least one positive and one negative summand.

Our aim is to construct sequences of indexes, alternating between negative and positive summands of different $x_i$'s. We have shown every basis element has at least one positive and one negative summand. With this in mind, we begin the construction of our sequences. 

Fix $s_0\ne 1$ so that $L_{s_0}$ is a positive summand of some $x_{j_0}$. Since $k>1$, then there exists $j_1\ne j_0$ such that $L_{s_0}$ is a negative summand of $x_{j_1}$. By our preceding discussion, there must exist a positive summand of $x_{j_1}$. So we can find $s_1\ne 1$ (since $L_1$ can only be a negative summand) such that $L_{s_1}$ is a positive summand of $x_{j_1}$. Thus $L_{s_1}$ is a negative summand of $x_{j_2}$ for some $j_2\ne j_1$. Again, there exists some $s_2\ne 1$ such that $L_{s_2}$ is a positive summand of $x_{j_2}$. Recursively, we can construct sequences of indexes $\{s_t\}$ and $\{j_t\}$ such that $s_t \ne 1$ and $j_t\ne j_{t+1}$ for all $t$, and $L_{s_t}$ is a positive summand of  $x_{j_{t}}$ and a negative summand of $x_{j_{t+1}}$. Since there are only finitely many $\{x_i\}$, the indexes $j_t$ must repeat at some point. Without loss of generality, assume $j_1$ is the first one that repeats, so our sequence is $\{j_1, j_2, j_3, \dots, j_n, j_1, \dots \}$, for some $n\geq 2$.

Let $y:=a_1 x_1 + \dots + a_kx_k $ be an element in $A_k$ that can be written as a positive linear combination of $L_t$'s. We show that $y $ is in the subspace generated by  $\{x_i\}_{i\ne j_1,\dots j_n}$ and  $x_{j_1}+ \dots  + x_{j_n}$. Since $L_{s_1}$ has multiplicity $a_{j_1}-a_{j_2}$ in $y$, it must happen that $a_{j_1}\geq a_{j_2}$. Now, $L_{s_2}$ has multiplicity $a_{j_2}-a_{j_3}$ in $y$, which implies $a_{j_2}\geq a_{j_3}$. Then we can obtain a sequence
\begin{align*}
	a_{j_1}\geq  a_{j_2}\geq \dots\geq a_{j_n}\geq a_{j_1},
\end{align*}
which implies $a_{j_1}=a_{j_{2}}=\dots = a_{j_n}$, as desired. 

Consequently, elements that can be written as a positive linear combination of $L_i$'s are contained in a subspace of dimension less than $k$, which contradicts our statement \eqref{A_k}. Hence $k=\frac{p-1}{2}$ and so 
\begin{align*}
	\rank(\mathcal{K(C)})\geq \rank(\tilde F(\mathcal{K(C)}) )=\dim(\tilde F(\mathcal{K(C)}_{\mathbb Q}) )=\dim(A_k)= \frac{p-1}{2},
\end{align*}
which finishes the proof.
\end{proof}

\medbreak 

Let $\mathcal C$ now be a symmetric fusion category with Verlinde fiber functor $F:\mathcal C\to \Ver_p$, and suppose $F$ is surjective.
We denote also by $F$ the induced $\mathbb Q$-algebra homomorphism $\mathcal{K(C)}_{\mathbb Q}\to 	\mathcal K(\Ver_p)_{\mathbb Q}$. Then $F(\mathcal{K(C)}_{\mathbb Q})$ is a subalgebra of $\mathcal{K}(\Ver_p)_{\mathbb Q}$, and we will show that it is exactly $\mathcal{K}(\Ver_p)_{\mathbb Q}$. 

\begin{remark}\label{Iso for K(Verp)}
	Consider the $\mathbb Q$-algebra $\mathbb Q[\mathbb Z_2]\cong \mathbb Q(\epsilon)/(\epsilon^2-1)$. Then we have an isomorphism of $\mathbb Q$-algebras 
	\begin{align}\label{Q(Z_2) isomorphism}
		\begin{aligned}
\mathbb Q(z+z^{-1})\otimes \mathbb Q[\mathbb Z_2]&\xrightarrow{\cong} \mathbb Q(z+z^{-1})\oplus \mathbb Q(z+z^{-1})\\
w\otimes (a+b\epsilon) &\mapsto ((a+b)w, (a-b)w),
		\end{aligned}
	\end{align}
for all $w\in \mathbb Q(z+z^{-1})$ and $a,b\in \mathbb Q$. Recall that by \cite[Proposition 3.3]{O2} we have an equivalence of symmetric fusion categories 
\begin{align*}
	\Ver_p\cong \Ver_p^+\boxtimes \operatorname{sVec}.
\end{align*}
Hence \eqref{Q(Z_2) isomorphism} induces an isomorphism of $\mathbb Q$-algebras
\begin{align}\label{iso for Ver_p}
	\mathcal K(\Ver_p)_{\mathbb Q} \cong \mathcal K(\Ver_p^+)_{\mathbb Q} \otimes \mathcal K(\operatorname{sVec})_{\mathbb Q}\cong \mathbb Q(z+z^{-1}) \otimes \mathbb Q[\mathbb Z_2]\xrightarrow{\cong} \mathbb Q(z+z^{-1})^{\oplus 2},
\end{align}
where the second isomorphism is given in \cite[Theorem 4.5 (iv)]{BEO}.

\end{remark}

By \eqref{iso for Ver_p}, we can identify $F(\mathcal{K(C)}_{\mathbb Q})$ with a $\mathbb Q$-subalgebra of $\mathbb Q(z+z^{-1})^{\oplus 2}$. Hence we start by looking at subalgebras of $\mathbb Q(z+z^{-1})^{\oplus 2}$. Recall we denote by $A_k$ the unique subextension of $\mathbb Q(z+z^{-1})$  such that $[A_k:\mathbb Q]=k$, see discussion at the beginning of the section.

\begin{lemma}\label{suablgebras}
	Subalgebras of $\mathbb Q(z+z^{-1})^{\oplus 2}$ of dimension greater than $\frac{p-1}{2}$ are of the form $\mathbb Q(z+z^{-1})\oplus A_k$ or $A_k\oplus \mathbb Q(z+z^{-1})$, where $k$ is a positive integer dividing $\frac{p-1}{2}$.
\end{lemma}

\begin{proof}
Let $\mathcal A$ be a subalgebra of $\mathbb Q(z+z^{-1})^{\oplus 2}$.
Suppose first that	 $\mathcal A$ has no nontrivial idempotents. Note that $\mathbb Q(z+z^{-1})^{\oplus 2}$ has no nilpotent elements and thus neither does $\mathcal A$.  Hence $\mathcal A$ is semisimple and so by Artin-Wedderburn's theorem it is isomorphic to a finite product of field extensions over $\mathbb Q$. Since $\mathcal A$ has no idempotents, this implies  that $\mathcal A$ is isomorphic to a field extension over $\mathbb Q$. 

Consider the projection map $p$ from $\mathbb Q(z+z^{-1})^{\oplus 2}$ to its first  summand, and let $q$ denote its restriction to $\mathcal A$. Then $\ker(q)=0$ or $\mathcal A$. If $\ker(q)=\mathcal A$ then elements in $\mathcal A$ are of the form  $(0,a)$, which is only possible for $a=0$ since $\mathcal A$ is a field. Hence if $\mathcal A \ne 0$ we must have $\ker(q)= 0$, that is, we have an injective map $A\hookrightarrow \mathbb Q(z+z^{-1})$, and so $\dim(\mathcal A)\leq \frac{p-1}{2}$.

	Suppose now that $\mathcal A$ contains a nontrivial idempotent $e$. Then $e$ is either $(1,0)$ or  $(0,1)$, and we have an isomorphism of $\mathbb Q$-algebras $$\mathcal A \cong e \cdot  \mathcal A \oplus (1-e)\cdot \mathcal A.$$ Hence $\mathcal A$ is a direct sum of $A_k\oplus A_l$ of subalgebras of $\mathbb Q(z+z^{-1})$, where $k,l\in \mathbb Z_{\geq 0}$ divide $\frac{p-1}{2}$. Lastly, note that if both $k,l<\frac{p-1}{2}$, then
\begin{align*}
\frac{p-1}{2}	<\dim(A_k\oplus A_l)=k+l\leq \frac{p-1}{4} +\frac{p-1}{4}=\frac{p-1}{2},
\end{align*}
a contradiction.
Hence we must have either $k=\frac{p-1}{2}$ or $l=\frac{p-1}{2}$, and thus either $A_k= \mathbb Q(z+z^{-1})$ or $A_l= \mathbb Q(z+z^{-1})$, as desired. 
\end{proof}

The proof of the following theorem follows analogous steps as the ones in the proof of Theorem \ref{rango p-1/2}.

\begin{theorem}\label{theo: p-1}Let $p\geq 5$ and 
	let $\mathcal C$ be a symmetric fusion category with Verlinde fiber functor $F:\mathcal C \to \Ver_p$. If $F$ is surjective, then $$\rank(\mathcal C)\geq p-1.$$
\end{theorem}

\begin{proof}
	Let $F:\mathcal C\to \Ver_p$ be as in the statement.  By Equation \eqref{iso for Ver_p}, we have an isomorphism of $\mathbb Q$-algebras
	\begin{align}\label{second identification}
		\mathcal K(\Ver_p)_{\mathbb Q} \cong \mathcal K(\Ver_p^+)_{\mathbb Q} \otimes \mathcal K(\operatorname{sVec})_{\mathbb Q}\cong \mathbb Q(z+z^{-1}) \otimes \mathbb Q[\mathbb Z_2]\xrightarrow{\cong} \mathbb Q(z+z^{-1})^{\oplus 2},
	\end{align}
	induced from the $\mathbb Q$-algebras isomorphism $\mathbb Q(z+z^{-1})\otimes \mathbb Q[\mathbb Z_2]\xrightarrow{\cong} \mathbb Q(z+z^{-1})^{\oplus 2}$, given in Equation \ref{Q(Z_2) isomorphism}. Hence, under this isomorphism we have identifications
	\begin{align}\label{identifications}
		\begin{aligned}
			L_{t+1}-L_{t-1}&=(z^t+z^{-t}, z^t+z^{-t}), \  \ \ \ \ \text{ for $t$ even, } 1< t <p-1,\\
			L_{t+1}-L_{t-1}&=(z^t+z^{-t}, -(z^t+z^{-t})),  \ \text{ for $t$ odd, } 1< t <p-1,\\
			-L_{p-2}&=(z^{p-1}+z^{-(p-1)}, z^{p-1}+z^{-(p-1)}), \text{ and}\\
			L_2&=(z+z^{-1}, -(z+z^{-1})).
		\end{aligned}
	\end{align}
	
	Since $\mathcal C$ is not super-Tannakian, by the proof of Theorem \ref{rango p-1/2} we know that the composition $$ \mathcal{K(C)} \xrightarrow{F} \mathcal K(\Ver_p)\cong \mathcal K(\Ver_p^+)\boxtimes \mathcal K(\text{sVec}) \xrightarrow{\operatorname{id}\boxtimes \operatorname{Forget}} \mathcal K(\Ver_p^+),$$  is surjective. Moreover, since we are assuming that $F:\mathcal C\to \Ver_p$ is surjective, $F(\mathcal{K(C)})$ cannot be equal to $\mathcal K(\Ver_p^+)$, and so 
	\begin{align*}
		\rank( F(\mathcal{K(C)}) ) >\frac{p-1}{2}.
	\end{align*}
	This together with Lemma \ref{suablgebras} implies that $F(\mathcal{K(C)}_{\mathbb Q})$ is identified with a subalgebra of the form $\mathbb Q(z+z^{-1})\oplus A_k$, for some $k$ that divides $\frac{p-1}{2}.$ Note that the rank of $F(\mathcal{K(C)})$  is equal to the dimension of $F(\mathcal{K(C)}_{\mathbb Q})$, and so we want to show that $k=\frac{p-1}{2}$.

	For objects $X\in \mathcal C$, their images under $F$ are objects in $\Ver_p$, and thus can be written as $\mathbb Z_{\geq 0}$ linear combinations of $L_t$'s. Then $F(\mathcal {K(C)}_{\mathbb Q})=\mathbb Q(z+z^{-1})\oplus A_k$ has a basis of elements of this form. That is, 
	\begin{align}\label{Q+A_k}
		\text{$\mathbb Q(z+z^{-1})\oplus A_k$ has a basis given by  $\mathbb Z_{\geq 0}$ linear combinations of  $L_t$'s.}
	\end{align}
	
	For the sake of contradiction, assume that $k<\frac{p-1}{2}$. 
	Since $F$ is surjective, we already know that $k>1$.
	
	Let $\mathcal O_1, \dots, \mathcal O_k$ and $x_1, \dots, x_k$ be as in \eqref{x_i basis}. Without loss of generality, we choose the labelling so that $z+z^{-1}=z^{p-1}+z^{-(p-1)}\in \mathcal O_1$. Then 
	\begin{align*}
		\{(z^i+z^{-i},0)\}_{i\text{ odd, }1 \leq i \leq p-1}\cup \{(0,x_i)\}_{i=1, \dots, k},
	\end{align*}
	is a basis for $\mathbb Q(z+z^{-1})\oplus A_k$. We start by writing this basis as a linear combination of $L_t$'s, following the identification \eqref{identifications}. Let $1<t< p-1$ even. Then we have
	\begin{align}\label{identification for z^t}
		\begin{aligned}
			2(z^t+z^{-t}, 0) &=(z^t+z^{-t}, z^t+z^{-t}) + (z^t+z^{-t}, -(z^t+z^{-t})) \\
			&=(z^t+z^{-t}, z^t+z^{-t}) + (z^{p-t}+z^{-(p-t)}, -(z^{p-t}+z^{-(p-t)}))\\
			&=L_{t+1}-L_{t-1}+L_{p-t+1}-L_{p-t-1},
		\end{aligned}
	\end{align}
	where the last equality is due to \eqref{identifications}, since $t$ is even and $p-t$ is odd. Analogously, 
	\begin{align*}
		2(z^{p-1}+z^{-(p-1)}, 0) &=-L_{p-2}+L_2.
	\end{align*}
	On the other hand, 
	\begin{align*}
		2(0,z^t+z^{-t}) &=(z^t+z^{-t}, z^t+z^{-t}) - (z^t+z^{-t}, -(z^t+z^{-t})) \\
		&=(z^t+z^{-t}, z^t+z^{-t}) - (z^{p-t}+z^{-(p-t)}, -(z^{p-t}+z^{-(p-t)}))\\
		&=L_{t+1}-L_{t-1}-L_{p-t+1}+L_{p-t-1}, \text{ and}\\
		2(0,z^{p-1}+z^{-(p-1)}) &=-L_{p-2}-L_2.
	\end{align*}
	Hence
	\begin{align}\label{identification for x_i}
		\begin{aligned}
			2(0,x_i)&=\sum\limits_{\substack{z^t +z^{-t}\in \mathcal O_i\\ t \text{ even}\\ 2\leq t <p-1}}2\left( 0,  z^t+z^{-t} \right) \\&= \sum\limits_{\substack{z^t +z^{-t}\in \mathcal O_i\\ t \text{ even}\\ 2\leq t <p-1}} (L_{t+1}-L_{t-1}-L_{p-t+1}+L_{p-t-1}), \ \ \ \text{ for $i\ne 1$,}
		\end{aligned}
	\end{align}
	and
	\begin{align*} 2(0,x_1)&=\sum\limits_{\substack{z^t +z^{-t}\in \mathcal O_1\\ t \text{ even}\\ 2\leq t <p-1}} (L_{t+1}-L_{t-1}-L_{p-t+1}+L_{p-t-1}) +(-L_{p-2}-L_2).
	\end{align*}
	
	Consider first the set $\{2(0, x_i)\}_{i=1, \dots, k}$ of basis elements of $A_k$. Let $1<s< p-1$. Note that $L_s$ appears with nonzero multiplicity in either two of these elements, with multiplicity $1$ and $-1$, respectively, or in none (since it may cancel out with itself). On the other hand, $L_{1}$ and $L_{p-1}$ appear in only one basis element (explicitly, the basis element $2(0,x_j)$ such that  $z^2+z^{-2}\in \mathcal O_j$), both with multiplicity $-1$. We will say $L_s$ is a  ``positive" summand of $x_i$ if it has multiplicity 1 in $x_i$, and is a ``negative" summand if it has multiplicity $-1$. We will also say that $L_s$ is an ``odd" summand if $1\leq  s \leq p-1$ is odd, and an ``even" summand when $s$ is even.
	
	Our assumption $k<\frac{p-1}{2}$ assures that every orbit has at least two elements. We thus claim that every $2(0,x_i)$ has at least one odd positive summand and one odd negative summand. In fact, this is clear for $i\ne 1$ since the number of odd positive summands in $2(0,x_i)$ is the same as the number of odd negative summands. For $2(0,x_1)$, the argument is the same as the one given in the proof of Theorem  \ref{rango p-1/2}. 
	
	Our aim is to construct sequences of indexes, alternating between negative odd and positive odd summands of different $2(0,x_i)$'s. We know every basis element has at least one positive and one negative odd summand. With this in mind, we begin the construction of our sequences. 
	
	Fix $1<s_0<p-1$ so that $L_{s_0}$ is an odd positive summand of some $(0,x_{j_0})$. Since $k>1$, then there exists $j_1\ne j_0$ such that $L_{s_0}$ is an odd negative summand of $(0,x_{j_1})$. By our preceding discussion, there must exist an odd positive summand $L_{s_1}$ of $x_{j_1}$, $s_1\ne 1$ ($L_{1}$ can only be a negative summand). 
	Thus $L_{s_1}$ must be an odd negative summand of some $(0,x_{j_2})$, with $j_2\ne j_1$. Again, there exists some $L_{s_2}$ odd positive summand of $(0,x_{j_2})$, $s_2\ne 1$.  Recursively, we can construct sequences of indexes $\{s_t\}$ and $\{j_t\}$ such that $1<s_t<p-1$ is odd,  $j_t\ne j_{t+1}$ for all $t$, and $L_{s_t}$ is an odd positive summand of  $x_{j_{t}}$ and an odd negative summand of $x_{j_{t+1}}$. Since there are only finitely many $\{(0,x_i)\}$, the indexes $j_t$ must repeat at some point. Without loss of generality, assume $j_1$ is the first one that repeats, so our sequence is $\{j_1, j_2, j_3, \dots, j_n, j_{n+1}=j_1, \dots \}$, for some $n\geq 2$.

	We now use our sequences of indexes to show that elements of $\mathbb Q(z+z^{-1})\oplus A_k$  that can be written as a positive linear combination of $L_t$'s are contained in a subspace of dimension strictly less than $\dim(\mathbb Q(z+z^{-1})\oplus A_k)=\frac{p-1}{2}+k.$
	
	Consider now the basis $$\{2(z^i+z^{-i},0)\}_{i\text{ odd, }1 \leq i \leq p-1}\cup \{2(0,x_i)\}_{i=1, \dots, k},$$ of $\mathbb Q(z+z^{-1})\oplus A_k$. Let 
	\begin{align}\label{equation for i}
		y:=\left( \sum\limits_{i \text{ odd }} a_i 2(z^i+z^{-i}), \sum\limits_{j=1}^k b_j 2x_j\right)\in \mathbb Q(z+z^{-1})\oplus A_k, 
	\end{align}
	so that $y$ that can be written as a positive linear combination of $L_t$'s under the identification \ref{identifications}. We show that $y $ is in the subspace generated by 
	\begin{align}\label{positive subspace}
		\{2(z^i+z^{-i},0)\}_{i\text{ odd, }1 \leq i \leq p-1}\cup \{2(0,x_i)\}_{i\ne j_1, \dots, j_n}\cup \{2(0,x_{j_1}+\dots +x_{j_n})\}.
	\end{align}
	We do this by computing the multiplicities of $L_{s_t}$ and $L_{p-s_t}$ in \eqref{equation for i}, for all $t=1, \dots, n$. Note that, if $L_s$ is an odd positive (respectively, negative) summand of $2(0,x_i)$, then $L_{p-s}$ is an even positive (respectively, negative) summand of $2(0,x_i)$, see Equation \eqref{identification for x_i}.

	Recall that $L_{s_1}$ is an odd positive summand of $2(0,x_{j_1})$, and an odd negative summand of $2(0,x_{j_2})$. Also, $L_{s_1}$ is a positive summand of $2(z^{s_1-1}+z^{-(s_1-1)},0)$ and a negative summand of $2(z^{s_1+1}+z^{-(s_1+1)},0)$, see Equation \eqref{identification for z^t}. Hence the multiplicity  of $L_{s_1}$ in $\eqref{equation for i}$ under the identifications \eqref{identification for z^t} and \eqref{identification for x_i} is
	\begin{align}\label{first multiplicity}
		b_{j_1}-b_{j_2}+a_{s_1-1}-a_{s_1+1}\geq 0.
	\end{align}
	On the other hand, $L_{p-s_1}$ is an even positive summand of $2(0,x_{j_1})$, and an even negative summand of $2(0,x_{j_2})$. But $L_{p-s_1}$ is a negative summand of $2(z^{s_1-1}+z^{-(s_1-1)},0)$ and a positive summand of $2(z^{s_1+1}+z^{-(s_1+1)},0)$, see Equation \eqref{identification for z^t}. Hence the multiplicity  of $L_{p-s_1}$ in $\eqref{equation for i}$  under the identifications \eqref{identification for z^t} and \eqref{identification for x_i} is
	\begin{align}\label{second multiplicity}
		b_{j_1}-b_{j_2}-a_{s_1-1}+a_{s_1+1}\geq 0.
	\end{align}
	Now, equations \eqref{first multiplicity} and \eqref{second multiplicity} imply that
	\begin{align*}
		b_{j_1}\geq b_{j_2}.
	\end{align*}
	Analogously, for $1\leq i \leq n$ we have that $L_{s_i}$ has multiplicity 
	\begin{align*}
		b_{j_{i}}-b_{j_{i+1}}+a_{s_i-1}-a_{s_i+1}\geq 0,
	\end{align*}
	in \eqref{equation for i}, and $L_{p-s_i}$ has multiplicity 
	\begin{align*}
		b_{j_{i}}-b_{j_{i+1}}-a_{s_i-1}+a_{s_i+1}\geq 0,
	\end{align*}
	which implies 
	\begin{align*}
		b_{j_{i}}\geq b_{j_{i+1}}.
	\end{align*}
	Hence, since $j_{n+1}=j_1$, we have that
	\begin{align*}
		b_{j_1}\geq  b_{j_2}\geq \dots\geq b_{j_n}\geq b_{j_{n+1}}=b_{1},
	\end{align*}
	which implies $b_{j_1}=b_{j_{2}}=\dots = b_{j_n}$, as desired. 
	
	Consequently, elements of $\mathbb Q(z+z^{-1})\oplus A_k$  that can be written as a positive linear combination of $L_t$'s are contained in the subspace \eqref{positive subspace}, which has dimension strictly less than $\dim(\mathbb Q(z+z^{-1})\oplus A_k),$ since $n\geq 2$. This contradicts \eqref{Q+A_k}, and the contradiction came from assuming $k<\frac{p-1}{2}$. Thus we must have $k=\frac{p-1}{2}$, and so 
	\begin{align}\label{iso for the image}
		F(\mathcal{K(C)}_{\mathbb Q})\cong \mathbb Q(z+z^{-1})\oplus A_k =\mathbb Q(z+z^{-1})\oplus \mathbb Q(z+z^{-1}),
	\end{align}
as $\mathbb Q$-algebras.
	Lastly, 
	\begin{align}\label{rank inequality}
		\rank(\mathcal{K(C)})\geq \rank(F(\mathcal{K(C)}) )= \dim(F(\mathcal{K(C)}_{\mathbb Q}))=\dim(\mathbb Q(z+z^{-1})\oplus \mathbb Q(z+z^{-1}))= p-1,
	\end{align}
	which finishes the proof.
\end{proof}

\begin{corollary}\label{groth rings of image}
	Let $p\geq 5$, and let $\mathcal C$ be a symmetric fusion category that is not super Tannakian. Let $F:\mathcal C\to \Ver_p$ be the Verlinde fiber functor. Then 
	\begin{align*}
		F(\mathcal{K(C)})= \mathcal K(\Ver_p) \ \ \text{ or } \ \ 	F(\mathcal{K(C)})= \mathcal K(\Ver_p^+).
	\end{align*}
	In particular,
	\begin{align*}
&&F(\mathcal{K(C)}_{\mathbb Q})\cong \mathbb Q(z+z^{-1}) ^{\oplus 2}&&\text{or} &&F(\mathcal{K(C)}_{\mathbb Q})\cong  \mathbb Q(z+z^{-1}).
	\end{align*}
\end{corollary}

\begin{proof}
The image of the functor $F:\mathcal C \to \Ver_p$ is a  fusion subcategory of $\Ver_p$, thus it can only be $\operatorname{Vec}, \operatorname{sVec}, \Ver_p^+$ or $\Ver_p$. The first two choices are not possible since we are assuming that $\mathcal C$ is not super Tannakian.
	
	Suppose first that the image is $\Ver_p$. Then $F$ is surjective, and so by the proof of  Theorem \ref{theo: p-1} we have that $\rank(F(\mathcal{K(C)})=p-1$, see Equation \eqref{rank inequality}, which implies $F(\mathcal K(C))=\mathcal K(\Ver_p)$. Also  $F(\mathcal{K(C)}_{\mathbb Q})\cong \mathbb Q(z+z^{-1})^{\oplus 2}$ by  Equation \eqref{iso for the image}.
	
	Suppose now that the image of $F$ is $\Ver_p^+$. Then the induced ring homomorphism $F:\mathcal{K(C)}\to \mathcal K(\Ver_p)$ has image contained in $\mathcal K(\Ver_p^+)$, which implies
	\begin{align}\label{ineq 1}
		\rank(F(\mathcal{K(C)})\leq \rank(\mathcal K(\Ver_p^+))=\frac{p-1}{2}.
	\end{align}
On the other hand, by the proof of Theorem \ref{rango p-1/2}, we know that the functor
	\begin{align*}
		\tilde F:= (\operatorname{id}\boxtimes \operatorname{Forget}) \circ F: \mathcal C\to \Ver_p^+,
	\end{align*}
induces a surjective homomorphism  of $\mathbb Q$-algebras 
\begin{align}\label{ineq 2}
	\tilde F : \mathcal{K(C)}_{\mathbb Q} \twoheadrightarrow \mathcal K(\Ver_p^+)_{\mathbb Q}\cong \mathbb Q(z+z^{-1}).
\end{align}
Then
\begin{align*}
\frac{p-1}{2}\geq	\rank(F(\mathcal{K(\mathcal C)}))=\dim(F(\mathcal K(\mathcal C)_{\mathbb Q})) \geq \dim(\tilde F(\mathcal K(\mathcal C)_{\mathbb Q}))=\frac{p-1}{2},
\end{align*}
	where the first inequality is due to Equation \eqref{ineq 1} and the last to Equation \eqref{ineq 2}. This implies that $\rank(F(\mathcal{K(C)}))=\frac{p-1}{2}$,  and thus $F(\mathcal{K(C)})=\mathcal K(\Ver_p^+)$. In particular, this implies that $F(\mathcal{K(C)}_{\mathbb Q})\cong \mathbb Q(z+z^{-1})$.
\end{proof}

\section{Some properties of the Adams operation}

Throughout this section, we assume $p>2$. 

\label{section: Adams operation}
 \subsection{Adams operation in $\Ver_p$}
Recall that the Adams operation of a symmetric fusion category $\mathcal C$ is defined as the ring endomorphism $\psi_2:\mathcal{K(C)} \to \mathcal{K(C)}$ given by $$\psi_2(X)=S^2(X)-\Lambda^2(X),$$ for all $X\in \mathcal{K(C)}$, see Section \ref{subsection:Adams operation}. We note that, for an object $X$ in $\mathcal C$, $\psi_2(X)$ is in $\mathcal{K(C)}$, and thus $\psi_2(X)$ is not necessarily a linear combination (of simple objects) with non-negative coefficients. 

In this section we study some  properties of the Adams operation in $\Ver_p$. We first give an explicit formula for the second Adams operation on simple objects $L_t$, $1\leq t \leq p-1$. We then use this formula to show that if an object $X$ in $\Ver_p$ is fixed by the Adams operation, then $X$ is in the abelian subcategory $\operatorname{Vec}$ generated by $\mathbf 1=L_1$.

\begin{remark}\label{image of psi 2 isodd}
	The image of $\psi_2:\mathcal{K}(\Ver_p)\to \mathcal{K}(\Ver_p)$ is contained in $\mathcal{K}(\Ver_p^+)$. In fact, 
	\begin{align*}
		L_t^2=\sum\limits_{s=1}^{\min(t,p-t)} L_{2s-1} = S^2(L_t) + \Lambda^2(L_t), 
	\end{align*}
	for all $i=1, \dots, p-1$, and so the multiplicity of even simples is zero in both $S^2(L_i)$ and $\Lambda^2(L_i).$
\end{remark}

Note that to compute $\psi_2$ on simple objects $L_r$ of $\Ver_p$, it is enough to compute it for $r$ odd, since $$\psi_2(L_r)=-\psi_2(L_{p-r}).$$ This follows from 
\begin{align*}
	\Lambda^2 L_r = L_{p-1}^2 \otimes S^2 L_{p-r}= S^2 L_{p-r},
\end{align*}
see \cite[Proposition 2.4]{EOV}.
\begin{example}
	In $\Ver_5$, $$\psi_2(L_3)=L_1-L_3=-\psi_2(L_2).$$
	In fact, we know that $\psi_2(L_3) = S^2(L_3)- \Lambda^2(L_3)$ and $L_3^2=L_1+L_3= S^2(L_3)+\Lambda^2(L_3)$. Hence there must exist $\epsilon_1, \epsilon_3 \in \{\pm 1\}$ such that $\psi_2(L_3)=\epsilon_1 L_1 + \epsilon_3 L_3$. Now, 
	\begin{align*}
		L_1 + 2\epsilon_1 \epsilon_3 L_3 + L_3^2 =\psi_2(L_3)^2= \psi_2(L_3^2)= (1+\epsilon_1)  L_1 + \epsilon_3 L_3,
	\end{align*} 
	and so $2= 1+\epsilon_1$, which implies $\epsilon_1=1$. It follows that $\epsilon_3=-1$.  
\end{example}

\begin{proposition}\label{adams op formula}
	The second Adams operation $\psi_2:\mathcal K(\Ver_p)\to \mathcal K(\Ver_p)$ is given by 
	\begin{align*}
	&	\psi_2(L_t)=\sum\limits_{s=1}^{\min(t,p-t)}(-1)^{s+1} L_{2s-1} \text{ for $t$ odd, } 1 \leq t \leq p-1, \text{ and} \\
&	\psi_2(L_t)=\sum\limits_{s=1}^{\min(t,p-t)}(-1)^{s} L_{2s-1} \text{ for $t$ even, } 1 \leq t \leq p-1.
	\end{align*}
\end{proposition}
\begin{proof}
Note that the second formula follows from the first by the equality $\psi_2(L_r)=-\psi_2(L_{p-r})$. We will use the  isomorphism of $\mathbb Q$-algebras $$\mathcal K(\Ver_p^+)_{\mathbb Q}\cong \mathbb Q(z+z^{-1}),$$ for $z$ a primitive $p$-th root of unity, see \cite[Theorem 4.5 (iv)]{BEO} and Section \ref{section: bounds}. Consider the basis $\left\{z^{2i}+z^{-2i}\right\}_{i=1, \dots, \frac{p-1}{2}}$ of $\mathbb Q(z+z^{-1})$.  Via this isomorphism, we have identifications
	\begin{align*}
		L_{2j+1}=\sum\limits_{l=1}^j (z^{2l}+z^{-2l})+1, \ \ \text{for } j=0, \dots, (p-3)/2,
	\end{align*}
	from which we compute
	\begin{align}\label{identification}
		\begin{aligned}
		&	z^t+z^{-t}\ =\  L_{t+1}-L_{t-1} \ \ \text{for } t \text{ even,} \  2\leq t<p-1, \text{ and}\\ 		& z^{p-1}+z^{-(p-1)} = -L_{p-2},
		\end{aligned}
	\end{align}
	see Section \ref{section: bounds} and the proof of Theorem \ref{rango p-1/2} for details. 
	
	Note that $(\psi_2)_{\mathbb Q} :\mathbb Q(z+z^{-1})\to \mathbb Q(z+z^{-1})$ maps $z+z^{-1}=z^{p-1}+z^{-(p-1)}$ to $z^2+z^{-2}$.
	In fact, we compute
	\begin{align*}
		\dim(S^2(L_{p-2}))&=\frac{(p-2)(p-1)}{2}=1\mod p,  \ \ \text{and} \\ \dim(\Lambda^2(L_{p-2}))&=\frac{(p-2)(p-3)}{2}=3\mod p.
	\end{align*}
	Since $S^2(L_{p-2})+\Lambda^2(L_{p-2})=L_{p-2}^2=L_1+L_3$, and $L_1$ can appear in either $S^2(L_{p-2})$ or $\Lambda^2(L_{p-2})$ but not both, then it must be the case that $S^2(L_{p-2})=L_1$ and $\Lambda^2(L_{p-2})=L_3$. Thus using identification \eqref{identification} we get 
	\begin{align}\label{psi_2 at z+z^-1}
	\psi_2(z^{p-1}+z^{-(p-1)})=-	\psi_2(L_{p-2})=-S^2(L_{p-2})+\Lambda^2(L_{p-2})= L_3-L_1=z^2+z^{-2},
	\end{align}
as desired. In particular, this implies 
\begin{align*}
	\psi_2(z^m+z^{-m})=z^{2m}+z^{-2m}, \ \ \ \text{for all } 1\leq m\leq p-1.
\end{align*}
We prove now our formulas for $\psi_2(L_t)$, $t$ odd, by induction. We do the case $1\leq t \leq \frac{p-1}{2}$ first. We know $\psi_2(L_1)=L_1$, so the formula works for $t=1$. 
Fix $1 < t \leq \frac{p-1}{2}$ odd, and suppose the formula is true for all odd $1\leq r<t$. Then
\begin{align*}
\psi_2(z^{t-1}+z^{-(t-1)})=z^{2(t-1)}+z^{-2(t-1)}= L_{2(t-1)+1}-L_{2(t-1)-1}=L_{2t-1}-L_{2t-3},
\end{align*} 
where in the second equality we are using the identification \eqref{identification}. Since $z^{t-1}+z^{-(t-1)}=L_{t}-L_{t-2}$, we compute using induction
\begin{align*}
	\psi_2(L_t)&=\psi_2(L_{t-2})+L_{2t-1}-L_{2t-3}\\
	&=\sum\limits_{s=1}^{t-2} (-1)^{s+1}L_{2s-1}+L_{2t-1}-L_{2t-3}\\
	&=\sum\limits_{s=1}^{t} (-1)^{s+1}L_{2s-1}=\sum\limits_{s=1}^{\min(t,p-t)} (-1)^{s+1}L_{2s-1},
\end{align*}
and so the formula holds for $t$. 

It remains to show that the formula holds for the case of odd  $\frac{p-1}{2}< t< p-1$ . We already computed $\psi_2(L_{p-2})=L_1-L_3$, so it works for $t=p-2$. Fix odd $\frac{p-1}{2}< t=p-l< p-1$ and assume the formula holds for all odd $t< r< p-1$. Since $z^{p-l+1}+z^{-(p-l+1)}= L_{p-(l-2)}-L_{p-l}$, we compute same as before
\begin{align*}
	\psi_2(L_{p-l})&=\psi_2(L_{p-(l-2)})-(L_{2l-1}-L_{2l-3})\\
	&=\sum\limits_{s=1}^{l-2} (-1)^{s+1} L_{2s-1}-L_{2l-3}+L_{2l-1}\\
	&=\sum\limits_{s=1}^{l} (-1)^{s+1} L_{2s-1}=\sum\limits_{s=1}^{\min(t,p-t)} (-1)^{s+1}L_{2s-1},
\end{align*}
as desired. 
\end{proof}

\medbreak
We now study objects in $\Ver_p$ that are fixed by the second Adams operation. 
For a simple object $X$ in a symmetric fusion category $\mathcal C$, we denote by $[Y:X]$ the multiplicity of $X$ in $Y$ for all $Y\in \mathcal C$.
\begin{corollary}\label{psi 2 en Verp}
	An object $X\in \Ver_p$ is fixed by $\psi_2$ if and only if $X\in \operatorname{Vec}$.
\end{corollary}

\begin{proof}
	Let $X\in \Ver_p$ such that $\psi_2(X)=X$ and let $a_1, \dots, a_{p-1}$ be non-negative integers such that $X=\sum\limits_{j=1}^{p-1} a_j L_j$. 
Since the image of $\psi_2$ is contained in $\mathcal K(\Ver_p^+)$ (see Remark \ref{image of psi 2 isodd})	then $a_2=a_4=\dots=a_{p-1}=0$.

Using the formulas from Proposition \ref{adams op formula} we compute
	\begin{align*}
		a_1 = [	X:L_1] =[	\psi_2(X):L_1] =\sum\limits_{j=1}^{\frac{p-1}{2}}  a_{2j-1}.
	\end{align*}
	Since $a_i$ is non-negative for all $i$ this implies $a_3=a_5=\dots=a_{p-2}=0$, as desired. 
\end{proof}

\begin{remark}
	The statement of Corollary \ref{psi 2 en Verp} only works for actual objects in $\Ver_p$, that is, objects that can be written as $\mathbb Z_{\geq 0}$ linear combinations of $L_1, \dots, L_{p-1}$. There can exist objects in $\mathcal K(\Ver_p)$ that are fixed by the second Adams operation but are not multiples of $L_1$. For example, consider $p=17$ and $L_5-L_7+L_9-L_{15}\in \mathcal K(\Ver_p)$. Then
	\begin{align*}
		\psi_2(L_5-L_7+L_9-L_{15})&=\psi_2(L_5)-\psi_2(L_7)+\psi_2(L_9)-\psi_2(L_{15})\\ &=(L_1-L_3+L_5-L_7+L_9)-(L_1-L_3+L_5-L_7+L_9-L_{11}+L_{13})+\\&+(L_1-L_3+L_5-L_7+L_9-L_{11}+L_{13}-L_{15})-(L_1-L_3)\\&=L_5-L_7+L_9-L_{15}
	\end{align*}
	and so $L_5-L_7+L_9-L_{15}$ is fixed by $\psi_2$ but is not in $\operatorname{Vec}$.
\end{remark}

\subsection{Powers of the Adams operation}

Recall that in this section we assume $p>2$. Here we classify symmetric fusion categories $\mathcal C$ such that $\psi_2^a=\psi_2^{a-1}$ for some $a\geq 1$. Namely, we show that such categories  are super-Tannakian and thus classified by group data, see \cite{DM, D, DG}. Moreover, we show that the case $a=1$ is only possible for the trivial category. That is, we prove that if $\psi_2=\text{Id}$ in $\mathcal{K(C)}$ then $\mathcal C=\operatorname{Vec}$.

	\begin{theorem}\label{remark on powers of psi}
	Let $p>2$ and let $\mathcal C$ be a non-super-Tannakian symmetric fusion category. If  the Adams operation $\psi_2: \mathcal{K(C)}\to \mathcal{K(C)}$ satisfies $\psi_2^a=\psi_2^b$ for some $a, b\in \mathbb Z_{\geq 0}$, then $2^a\equiv \pm 2^b \mod p$.
\end{theorem}

\begin{proof}
	Consider the fiber functor $F:\mathcal C \to \Ver_p$; we denote also by $F$ the induced ring homomorphism $\mathcal{K(C)}\to \mathcal K(\Ver_p)$. 
	Suppose now that $a>1$.  Since $F$ preserves the symmetric structure, we have that 
	\begin{align*}
		\psi_2^{a}(F(X)) = F(\psi_2^{a}(X))=F(\psi_2^{b} (X)) =\psi_2^{b}(F(X)), \text{   for all } X\in \mathcal{K(C)}. 
	\end{align*}
	That is, $\psi_2^{a}=\psi_2^{b}$ on the image of $\mathcal{K(C)}$ under $F$. Suppose $\mathcal C$ is not super  Tannakian. In particular, this implies $p>3$, since for $p=3$ all symmetric fusion categories are super Tannakian. By Corollary \ref{groth rings of image}, we know that $F(\mathcal{K(C)}_{\mathbb Q})$ is isomorphic as a $\mathbb Q$-algebra to $\mathbb Q(z+z^{-1})$ or $\mathbb Q(z+z^{-1})^{\oplus 2}$. Recall that $\psi_2(z+z^{-1})=z^2+z^{-2}$, see \eqref{psi_2 at z+z^-1}. Then $(\psi_2^a)_{\mathbb Q}= (\psi_2^{b})_{\mathbb Q}$ in $F(\mathcal{K(C)}_{\mathbb Q})$ would imply that $z^{2^a}+z^{-2^a}=z^{2^{b}}+z^{-2^{b}}$, and so $2^{a}\equiv \pm 2^{b} \mod p$.
\end{proof}

\begin{corollary}\label{tannakian}
	Let $p> 2$ and let $\mathcal C$ be a symmetric fusion category. If  $\psi_2^a=\psi_2^{a-1}$ for some $a\in \mathbb Z_{\geq 1}$, then $\mathcal C$ is super-Tannakian.
\end{corollary}
\begin{proof}
	By Theorem \ref{remark on powers of psi}, if $\mathcal C$ is not super Tannakian then $2^{a}\equiv \pm  2^{a-1}\mod p$, which implies $2\equiv \pm 1 \mod p$, a contradiction.
\end{proof}

\begin{remark}\label{remark:psi = id}
	Let $p>2$. If $\psi_2:\mathcal K(C)\to \mathcal{K(C)}$ satisfies $\psi_2=\text{id}$, then  $\mathcal C$ is actually Tannakian. In fact, since $F$ preserves the symmetric structure, we have that 
	\begin{align*}
		\psi_2(F(X)) = F(\psi_2(X))=F(X), 
	\end{align*}
	for all $X\in \mathcal C$. Since $F(X) \in \Ver_p$ and $\psi_2$ fixes $F(X)$, then by Corollary \ref{psi 2 en Verp} we have that $F(X)\in \text{Vec}$ for all $X\in \mathcal C$, and so $\mathcal C$ is Tannakian, as desired.  
\end{remark}

	However, we show next that $\psi_2=\text{id}$ is only possible for $\mathcal C=\operatorname{Vec}$.

\begin{theorem}\label{psi not identity} Let $p \ne  2$. 
	If $\mathcal C$ is a non-trivial symmetric fusion category  then $\psi_2$ is not the identity. 
\end{theorem}
\begin{proof}
Let $p>2$. Suppose $\psi_2$ is the identity in $\mathcal C$. In Remark \ref{remark:psi = id}  we showed that $\mathcal C$ is Tannakian, and thus equivalent to $\Rep_{\textbf{k}}(G)$ for a finite group scheme $G$. A classification of finite group schemes $G$ such that $\Rep_{\textbf{k}}(G)$ is semisimple is given by Nagata's theorem  \cite[IV, 3.6]{DG}; thus Remark \ref{remark:psi = id} yields a
classification of symmetric fusion categories such that $\psi_2=\text{Id}$. Namely, any such category is an equivariantization (see \cite[Section 4]{DGNO}) of a pointed category such that the group of simples is an abelian $p$-group (see e.g. \cite[8.4]{EGNO}), by the action of a group $H$ of order relatively prime to $p$.
Suppose $H$ is non-trivial and consider the subcategory $\Rep_{\textbf{k}}(H)$ of $\mathcal C$.  

The Adams operation acts on $\Rep_{\textbf{k}}(H)$ by mapping a character $\chi(g)$ to $\chi(g^2)$ for all $g\in H$. Thus if $\psi_2=\text{Id}$ then 
$$\chi(g^{2})=\chi(g) \text{ for all } g\in H.$$
Hence for all $g$ in $H$, $g$ is conjugate to $g^{2}$. So if $|H|$ is even, there is an element $h\in H$ of order 2, which is conjugated to $h^2=1$, a contradiction. 

Suppose then $|H|$ is odd. We have that for all $g$ in $H$ there exists some $h\in H$ such that $g^{2}=hgh^{-1}$ and so $$g=hgh^{-1}g^{-1}.$$ Thus $g$ is in the commutator subgroup of $H$, and so $H\subseteq [H,H]$. This contradicts the Feit-Thompson theorem, which states that every finite group of odd order is solvable. So $H$ must be trivial.
	
	We thus have that $\mathcal C$ is a pointed category associated with an abelian $p$-group $P$. Hence $\psi_2$ maps $g\mapsto g^2$ for all $g\in P$  and so $g=g^2$ for all $g\in P$. Then $P$ is trivial and $\mathcal C$ is equivalent to $\vect$. 
	
	The result also holds in characteristic $0$, since the Adams operation acts on $\Rep_{\textbf{k}}(G)$ by mapping a character $\chi(g)$ to $\chi(g^2)$ for all $g\in G$.
\end{proof}

\begin{remark}
The hypothesis of $\mathcal C$ being finite is neccesary. Indeed, in \cite{HSS} it is shown that in any characteristic there is a semisimple, but not finite, symmetric category, known as the Delannoy category, for which all Adam operations are the identity.
\end{remark}

\subsection{Symmetric fusion categories with exactly two self-dual simple objects}

In this subsection we prove some general properties of the Adams operation $\psi_2:\mathcal{K(C)}\to \mathcal{K(C)}$ for the case when $\mathcal C$ is a symmetric fusion category with exactly two self-dual simple objects. These results will be useful for the classification of symmetric fusion categories of ranks 3 and 4 in Sections \ref{section: rank 3} and \ref{section: rank 4}. In particular, we show that if $\psi_2$ is an automorphism then it has even order, see Theorem \ref{even order}.

Recall that throughout this section we assume $p> 2$.

\begin{lemma}\label{equiv 1 mod 2}
	Let $\mathcal C$ be a symmetric fusion category with exactly two self-dual simple objects $\mathbf{1}$ and $Y$. Then 
	\begin{align*}
		[\psi_2^{2k+1}(Y):\mathbf{1}]\equiv 1\mod 2
	\end{align*}
	for all $k\geq 0$. 
\end{lemma}

\begin{proof}
	
	We proceed by induction on $k$.
	Let $\mathbf 1, Y, X_1, X_1^*, \dots, X_n, X_n^*$ denote the simple objects in $\mathcal C$. 
	Since $\psi_2(Y)\equiv Y^{2} \mod 2$ and $[Y^{2}:\mathbf{1}]=1$ then $[\psi_2(Y):\mathbf{1}]\equiv 1 \mod 2$, which proves the base case. 
	
	Fix $k>1$ and suppose that  
	\begin{align*}
		[\psi_2^{2l+1}(Y):\mathbf{1}]\equiv 1 \mod 2,  \text{  for all } l<k.
	\end{align*}
	We want to show this also holds for $l=k$. Since $Y$ is self-dual then $[Y^2:X_i]=[Y^2:X_i^*]$ for all $i=1, \dots, n$, and so
	\begin{align*}
		\psi_2(Y)\equiv Y^2 \equiv
		\mathbf{1} +\sum\limits_{i=1}^n [Y^{2}:X_i] \ (X_i+X_i^*)  + [Y^2:Y] \ Y \mod 2 , 
	\end{align*}
	for all $i=1, \dots, n$. Applying $\psi_2^{2k}$ on both sides of the previous equation we get 
	\begin{align}\label{eq:1}
		\psi_2^{2k+1}(Y)&\equiv \mathbf{1} +\sum\limits_{i=1}^n [Y^{2}:X_i] \left(\psi_2^{2k}(X_i) +\psi_2^{2k}(X_i^*)\right)  + [Y^2:Y]\psi_2^{2k}(Y)  \mod 2.
	\end{align}
	Recall that $\psi_2$ commutes with duality. That is, $\psi_2(X_i)^*=\psi_2(X_i^*)$ and so 
	\begin{align}\label{eq:2}
		[\psi_2^{l}(X_i):\mathbf{1}]=[\psi_2^l(X_i)^*:\mathbf{1}]\equiv [\psi^l_2(X_i^*):\mathbf{1}] \mod 2,
	\end{align}
	for all $l\geq 1$. From Equations \eqref{eq:1} and \eqref{eq:2} we get
	\begin{align*}
		[\psi_2^{2k+1}(Y):\mathbf{1}]&\equiv  1+ 2\sum\limits_{i=1}^n [Y^{2}:X_i] \  [\psi_2^{2k}(X_i):\mathbf{1}] + [Y^2:Y]\ [\psi_2^{2k}(Y):\mathbf{1}] \mod 2\\
		&\equiv 1+ [Y^2:Y]\ [\psi_2^{2k}(Y): \mathbf{1}] \mod 2.
	\end{align*}
	Analogously, 
	\begin{align*}
		[\psi_2^{2k}(Y):\mathbf{1}] 
		&\equiv 1 + [Y^2:Y] [\psi_2^{2k-1}(Y):\mathbf{1}] \mod 2\\
		&\equiv 1 + [Y^2:Y] \mod 2,
	\end{align*}
	since we assumed $[\psi^{2k-1}_2(Y):\mathbf{1}]\equiv 1 \mod 2$. Hence
	\begin{align*}
		[\psi_2^{2k+1}(Y):\mathbf{1}]
		&\equiv 1+ [Y^2:Y]\left( 1 + [Y^2:Y]  \right) \mod 2\\
		&\equiv 1 + [Y^2:Y]+ [Y^2:Y]^2  \mod 2\\
		&\equiv 1 \mod 2,
	\end{align*}
	as desired. 
\end{proof}


The following is a direct application of Lemma \ref{equiv 1 mod 2}.
\begin{theorem}\label{even order}
	Let $\mathcal C$ be a symmetric fusion category with exactly two self-dual simple objects. If $\psi_2$ is an automorphism  of $\mathcal{K(C)}$ then it has even order. 
\end{theorem}

\begin{proof}
	Let $k\geq 0$ such that $\psi_2^k=\text{Id}$, and let $\textbf 1, Y$ denote the self-dual objects. By Theorem \ref{equiv 1 mod 2} the multiplicity of $\mathbf{1}$ in $\psi_2^k(Y)=Y$ is positive whenever $k$ is odd, and thus  $k$ must be even.
\end{proof}

In the following proposition we restrict to the case when $\psi_2:\mathcal{K(C)}\to \mathcal{K(C)}$ has trivial image. 

\begin{proposition}\label{image is Q}
	Let $\mathcal C$ be a symmetric fusion category with exactly two self-dual simple objects $\mathbf{1}$ and $Y$. If $\text{Im}({\psi_2})\cong \mathbb Z$ then $Y^2=\mathbf{1}$ and $\psi_2(Y)=-\mathbf{1}$. Moreover, $[XX^*:Y]=1$ and $[XY:Y]=0$ for all non-self-dual simple $X$.
\end{proposition}

\begin{proof}
	Let $X$ be a non-self-dual simple object. Since $\mathbf{1}$ is not a summand of $X^2$ then $\mathbf{1}$ has multiplicity zero in $\psi_2(X)$. On the other hand, the multiplicity of  $\mathbf{1}$  in $Y^2$ is $1$ and so its coefficient  in $\psi_2(Y)$ is $\pm 1$. Thus if ${\psi_2}$ has trivial image we get that $\psi_2(X)=0$ for all non-self-dual simple $X$ and $\psi_2(Y)=\epsilon$, where $\epsilon=\pm 1$. Hence
	\begin{align*}
		\mathbf{1}=\psi_2(Y)^2=\psi_2(Y^2)= \sum\limits_{\text{simple } Z} [Y^2:Z]\  \psi_2(Z)= \left( 1+ \epsilon \cdot [Y^2: Y]\right) \cdot \mathbf{1},
	\end{align*}
	which implies
	\begin{align}\label{eq Y^2 1}
		[Y^2: Y]=0.
	\end{align}
	Similarly, 
	\begin{align}\label{eq Y^2 2}
		0=\psi_2(XY)= \sum\limits_{\text{ simple } Z} [XY:Z] \psi_2(Z)=  \epsilon \cdot  [XY: Y] \cdot \mathbf{1},
	\end{align}
	and thus 
	\begin{align*}
		[XY:Y]=0, \text{ for all non-self-dual simple } X.
	\end{align*}
	From the fusion rule $[XY:Y]=[Y^2:X]$ and Equations \eqref{eq Y^2 1} and \eqref{eq Y^2 2} we conclude $Y^2=\mathbf{1}$. Lastly, note that
	\begin{align*}
		0 =\psi_2(XX^*)=\mathbf{1} + \epsilon \cdot [XX^*: Y]\cdot \textbf 1,
	\end{align*}
	and thus we must have $\epsilon =-1$ and $[XX^*:Y]=1$. 
\end{proof}

\section{Rank 3 symmetric fusion categories}\label{section: rank 3}

In this section we classify symmetric fusion categories $\mathcal C$ of rank 3, making use of the properties of the Adams operation. Namely, we show that $\mathcal C$ is equivalent to one of the following:
\begin{itemize}
	\item If $p=2$, $\mathcal C \cong \Rep_{\mathbf k}(\mathbb Z_ 3)$.
	\item If $p=3$, $\mathcal C \cong\vect_{\mathbb Z_3}^{\mathbb Z_2}$ or $\mathcal C \cong\vect_{\mathbb Z_3}$ .
	\item If $p=7$, $\mathcal C \cong \Rep_{\mathbf k}(S_3)$, $\mathcal C \cong \Rep_{\textbf k}(\mathbb Z_3)$ or $\mathcal C\cong \Ver_7^+$.
	\item If $p=5$ or $p>7$, $\mathcal C \cong \Rep_{\mathbf k}(S_3)$ or  $\mathcal C \cong \Rep_{\textbf k}(\mathbb Z_3)$.
\end{itemize}

We note that if $\mathcal C$ is non-super-Tannakian, then by Theorem \ref{rango p-1/2} we know that $p\leq 7$ . Hence the only possibilities are $p=5$ or $7$, since in the cases $p=2$ or $p=3$ the category would be super-Tannakian. 

We make use of the parametrization of self-dual based rings of rank 3 as given in \cite{O}. Let $k,l,m,n$ be non negative integers satisfying
\begin{align}\label{rank 3}
	k^2 + l^2 = kn + lm + 1,
\end{align}
and consider the ring $K(k,l,m,n)$ with basis $1,X,Y$ and multiplication rules
\begin{align}\label{rank 3 mult}
	&X^2 = 1 + mX + kY, &Y^2 = 1 + lX + nY, &&XY = Y X = kX + lY.
\end{align}
Note that we have a  based ring isomorphism $K(k, l, m, n) \cong K(l, k, n, m)$ obtained by the interchange $X \leftrightarrow Y$. Hence we will assume $ l \geq k$. By \cite[Proposition 3.1]{O} any unital based ring of rank 3 is isomorphic to either $K(k, l, m, n)$ or $K(\mathbb Z_3)$, where $\mathcal K (\mathbb Z_3)$ denotes the group algebra of the group $\mathbb Z/3\mathbb Z$, which has rank $3$ with basis given by the group elements.

Recall that a fusion category is \emph{integral} if the Frobenius Perron dimension of every simple object $X$ is an integer. We have the following result. 

\begin{theorem}\label{rank 3 integer}
	There is an integral symmetric fusion category $\mathcal C$ with Grothendieck ring $K(k,l,m,n)$ with $l \geq k$ if and only if $(k,l,m,n)=(0,1,0,1)$ and $p\geq 3$. Moreover, in such case
	\begin{enumerate}
		\item $\mathcal C\cong \Rep(S_3)$ if $p>3$, or
		\item $\mathcal C \cong \vect_{\mathbb Z_3}^{\mathbb Z_2}$ if $p=3$.
	\end{enumerate}
\end{theorem}

\begin{proof}
	Suppose $\mathcal C$ is an integral fusion category with Grothendick ring given by $K(k,l,m,n)$ with  $l\geq k$. Taking Frobenius-Perron dimension on the multiplication rules for $X^2$ and $XY$  (Equation \eqref{rank 3 mult}), we get that
	\begin{align*}
	\FPdim(X)^2=1+m\FPdim(X)+ k\FPdim(Y)&&\text{and} &&1=\frac{k}{\FPdim(Y)}+\frac{l}{\FPdim(X)},
	\end{align*}
respectively. 
From the first equality we deduce that $\FPdim(X)$ and $\FPdim(Y)$ are coprime. 
Hence the second equality is only possible if $k=0$ and $l=\FPdim(X)$. So the multiplication rules are 
	\begin{align*}
		&X^2 = 1 + mX, &Y^2 = 1 + \FPdim(X)X + nY, &&XY = Y X =  \FPdim(X)Y.
	\end{align*}
Taking Frobenius-Perron dimension on both sides of  the equation for $X^2$ we get that $\FPdim(X)^2-m\FPdim(X)-1=0$, which implies
\begin{align*}
	\FPdim(X)=\frac{m + \sqrt{m^2+4}}{2}.
\end{align*}
	Since $\FPdim(X)$ is an integer, we need for $ m^2+4$ to be a square. It is easy to check this is only possible for $m=0$, and so the multiplication rules are
		\begin{align*}
		&X^2 = 1, &Y^2 = 1 + X + nY, &&XY = Y X = Y.
	\end{align*}
	Taking Frobenius-Perron dimension on both sides of the equation for $Y^2$ we get that 
	\begin{align*}
		\FPdim(Y)=\frac{n + \sqrt{n^2+8}}{2}.
	\end{align*}
We thus need $n^2+8$ to be a square, which is only possible for $n=1$. Hence
		\begin{align}\label{dim eq}
		&X^2 = 1, &Y^2 = 1 + X + Y, &&XY = Y X = Y.
	\end{align}
So far we have showed that if $\mathcal C$ is integral, then it must have fusion rules as above. We have not used the assumption that $\mathcal C$ is symmetric yet. We will use it in what follows to complete the proof.

We look at the case $p>3$ first.  From Equations \eqref{dim eq} we get that $\dim(X)=1$ and $\dim(Y)=1$ or $2$, and thus $\dim(\mathcal C)=3$ or $6$. Since $p>3$, then $\dim(\mathcal C)\ne 0$ and thus $\mathcal C$ is non-degenerate. Hence we can lift $\mathcal C$ to a symmetric fusion category $\widetilde{\mathcal C}$ over a field $\textbf{f}$ in characteristic zero, which has the same Grothendick ring as $\mathcal C$, see \cite[Subsection 9.16]{EGNO} and \cite{E}. 
Thus $\widetilde{\mathcal{C}}$ is equivalent to $\Rep_{\textbf f}(S_3)$ (see  \cite[Section 8.19]{D}), and so by uniqueness of the lifting we get that $\mathcal C$ is equivalent to $\Rep_{\textbf k}(S_3)$. 

For the case $p=3$ we have that $\mathcal C$ contains a copy of $\Rep(\mathbb Z_2)$. Doing de-equivariantization by $\mathbb Z_2$ we obtain a symmetric fusion category $\mathcal C^{\mathbb Z_2}$ of dimension $3$ \cite[Section 2.7]{EGNO}. Hence $\mathcal C^{\mathbb Z_2}$ is a symmetric pointed category and must be equivalent to $\operatorname{Vec}_{\mathbb Z_3}$. There is only one action of $\mathbb Z_2$ on $\mathbb Z_3$, and so doing equivariantization by $\mathbb Z_2$ gives us back $\mathcal C$. Hence
$\mathcal C\cong \text{Vec}_{\mathbb Z_3}^{\mathbb Z_2}$.

On the other hand, for $p=2$ there is no category realizing these fusion rules. In fact, in such case we would have that the category is Tannakian, and so we have a fiber functor $F:\mathcal C\to \operatorname{Vec}$.
From Equation $\eqref{dim eq}$ we know that $X$ is invertible and $Y$ is not, and thus the same is true for $F(X)$ and $F(Y)$, respectively. Thus $d:=\dim(F(Y))$ is not 1 and must satisfy $d^2=2+d=d$. The only other possible solution is thus $d=0$, and since $F$ preserves dimensions we obtain $\dim(Y)=0$, which is not possible. 
\end{proof}

\begin{theorem}\label{theo: rank 3}
If $\mathcal C$ is a non-integral symmetric fusion category with Grothendieck ring $K(k,l,m,n)$ then $p=7$ and $\mathcal C=\Ver_7^+$.
\end{theorem}

\begin{proof}
Note that $p\ne 2, 3$ since in that case $\mathcal C$ is super-Tannakian and thus integral. Moreover, we also know that $p\leq 7$ by Theorem \ref{rango p-1/2}. However, our proof does not require use of this fact.

Consider the Adams operation $\psi_2: K(k,l,m,n)\to K(k,l,m,n)$. Since $X$ is self-dual, either $S^2(X)$ or $\Lambda^2(X)$ contains a copy of the unit object (but not both). The same is true for $Y$. Hence  the multiplicity of $\mathbf 1$ in both $\psi_2(X)$ and $\psi_2(Y)$ is 1 or -1, and so
\begin{align*}
	&\psi_2(X)=\epsilon_1 \textbf 1+ \alpha X + \beta Y, &\psi_2(Y)=\epsilon_2 \textbf 1 + \gamma X + \delta Y,
\end{align*}
for some $\epsilon_1, \epsilon_2\in \{1,-1\}$ and $\alpha, \beta,\gamma, \delta \in \mathbb Z$.

Since $\mathcal{K(C)}$ is a based ring of rank $3$ then $\mathcal{K(C)}_{\mathbb Q}:=\mathcal{K(C)}\otimes \mathbb Q$ is a semisimple commutative $\mathbb Q$-algebra of dimension 3, see \cite[Corollary 3.7.7]{EGNO}. Hence we have three distinct possibilities for $\mathcal{K(C)}_{\mathbb Q}$ as a $\mathbb Q$-algebra, and we proceed by looking at them separetly. 

$\blacktriangledown$ Case 1: $\mathcal{K(C)}_{\mathbb Q}\cong \mathbb Q \oplus \mathbb Q \oplus \mathbb Q$. In this case the homomorphism $\FPdim_{\mathbb Q} : \mathcal{K(C)}_{\mathbb Q} \to \mathbb R$ can only have  rational image. Since $\FPdim(X)$ is an algebraic integer for all $X\in \mathcal C$, this implies that $\FPdim(X)$ is an integer for all $X\in \mathcal C$, and thus $\mathcal C$ is integral.

$\blacktriangledown$ Case 2: $\mathcal{K(C)}_{\mathbb Q}\cong \mathbb Q \oplus \mathbb Q(\sqrt{m})$ for some $m\in \mathbb Z$. Since $(\psi_2)_{\mathbb Q}$ is an endomorphism of  $\mathcal{K(C)}_{\mathbb Q}\cong \mathbb Q \oplus \mathbb Q(\sqrt{m})$ mapping $(1,1)\mapsto (1,1)$, then it is either an automorphism of order $1$ or $2$, or has image the diagonal copy of $\mathbb Q$ inside $ \mathbb Q \oplus \mathbb Q(\sqrt{m})$. Hence we have three possibilities for $\psi_2:\mathcal{K(C)}\to \mathcal{K(C)}$: it satisfies $\psi_2=\text{Id}$, $\psi_2^2=\text{Id}$ or $\text{Im}(\psi_2)= \mathbb Z$. 
We show none of these are possible. 

We assume first that $\text{Im}(\psi_2)=\mathbb Z$. So we have $\psi_2(X)=\epsilon_1\textbf 1$ and $\psi_2(Y)=\epsilon_2\textbf 1$. The equalities
\begin{align*}
1=[\psi_2(X)^2:\mathbf 1]	=[\psi_2(X^2):\mathbf 1]&=[\mathbf 1+m\psi_2(X)+k\psi_2(Y):\mathbf 1]=1+m\epsilon_1+k\epsilon_2,
\end{align*}
imply that $\epsilon_1$ and $\epsilon_2$ must have opposite signs and that $m=k$. The analogous computation with $Y$ shows that $l=n$.  Moreover, since
\begin{align*}
\epsilon_1\epsilon_2=[\psi_2(X)\psi_2(Y):\mathbf 1]=[\psi_2(XY):\mathbf 1]=[k\psi_2(X)+l\psi_2(Y):\textbf 1]=k\epsilon_1+l\epsilon_2,
\end{align*}
  we have that $l=k\pm 1$. Recall that we are assuming $l\geq k$ and thus $l=k+1$. Then $$X^2=1+kX+(k+1)Y,$$
  and so 
\begin{align*}
	&X^2\equiv 1+X \mod 2&&\text{or} &&X^2 \equiv 1+Y \mod 2,
\end{align*}
which contradicts $X^2\equiv \psi_2(X) \equiv 1\mod 2$. So the case  $\text{Im}(\psi_2)=\mathbb Z$ is not possible.

On the other hand, $\psi_2=\text{Id}$ is not possible by Theorem \ref{psi not identity}. Thus it remains to show that we cannot have $\psi_2^2=\text{Id}$ either. Suppose $\psi_2^2=\text{Id}$. Then 
\begin{align*}
&	0=[\psi_2^2(X):\mathbf 1]=\epsilon_1+\alpha\epsilon_1+\beta\epsilon_2 &\text{and} &&0=[\psi_2^2(Y):\mathbf 1]= \epsilon_2+\gamma \epsilon_1+\delta\epsilon_2,
\end{align*}
which imply 
\begin{align}\label{eq}
&&	\beta=-(\alpha+1)\epsilon_1\epsilon_2 &&\text{and} &&\gamma=-(\delta+1)\epsilon_1\epsilon_2.
\end{align} Also
\begin{align*}
	&0=[\psi^2_2(X):Y]=\beta(\alpha+\delta) &\text{and} &&0=[\psi_2^2(Y):X]=\gamma(\alpha+\delta).
\end{align*}
If $\alpha+\delta\ne 0$ then $\beta=\gamma=0$, which implies $\alpha=\delta=-1$.  That is, 
\begin{align*}
	&\psi_2(X)=\epsilon_1\textbf 1-X &&\text{and} &&\psi_2(Y)=\epsilon_2\textbf 1-Y.
\end{align*}
We compute
\begin{align*}
	&\psi_2(X)\psi_2(Y)=\epsilon_1\epsilon_2\textbf 1 +(k-\epsilon_2)X+(l-\epsilon_1)Y,\\
	&\psi_2(XY)=k\psi_2(X)+l\psi_2(Y)=(k\epsilon_1+l\epsilon_2)\textbf 1 - kX -lY.
\end{align*}
Since the right-hand sides of the previous equations must be equal, we have that $k-\epsilon_2= -k$ and so $2k=\pm 1$, which is not possible. 

 Hence we should have $\alpha+\delta=0$, and using this together with \eqref{eq} we get
\begin{align}\label{cong}
	&&\alpha\equiv \delta \mod 2, &&\beta=\gamma \mod 2 &&\text{and} &&\alpha\equiv \beta+1 \mod 2.
\end{align}
 On the other hand, from the equation
 \begin{align*}
 m\alpha +k\gamma =[\psi_2(X^2):X]=[\psi_2(X)^2:X]=2\epsilon_1\alpha + \alpha^2 m +2\alpha\beta k+\beta^2l,
 \end{align*}
we obtain
\begin{align*}
	m\alpha +k\gamma  \equiv m\alpha  +l\beta \mod 2.
\end{align*}
This together with congruences \eqref{cong} gives $k\beta\equiv l\beta\mod 2$. Lastly, 
\begin{align*}
	&[\psi_2(XY):X]=k\alpha+l\gamma,\\
	&[\psi_2(X)\psi_2(Y):X]=\epsilon_1\gamma+\epsilon_2\alpha+m\gamma\alpha+k\alpha\delta+k\gamma\beta+l\beta\delta,
\end{align*}
and so 
\begin{align*}
	\gamma+\alpha+m\gamma\alpha+k\alpha\delta+k\gamma\beta+l\beta\delta\equiv k\alpha + l\gamma \mod 2.
\end{align*}
Note that $\alpha\gamma\equiv 0\equiv \beta\delta \mod 2$, see \eqref{cong}. Hence the equation above is
\begin{align*}
	\gamma+\alpha+k\alpha\delta+k\gamma\beta\equiv k\alpha + l\gamma \mod 2.
\end{align*}
Now, using the congruences in \eqref{cong} and $k\beta\equiv l \beta \mod 2$ in the equation above, we obtain $1\equiv  0 \mod 2$. Thus $(\psi_2)_{\mathbb Q}$ cannot be an automorphism of order 2 and this case is not possible. 

$\blacktriangledown$ Case 3: $\mathcal{K(C)}$ is a field extension of degree 3 over $\mathbb Q$. In this case $(\psi_2)_{\mathbb Q}$ must be an automorphism of order 3, since it cannot be the identity by Theorem \ref{psi not identity}. Note that Theorem \ref{remark on powers of psi} implies $p=3$ or $7$. We show $p=7$ and $\mathcal C\cong \Ver_7^+$.
Since $\psi_2$ has order 3, then  $X, \psi_2(X)$ and $\psi_2^2(X)$ are distinct roots of the minimal polynomial of $X$, given by $$m_X(t)=t^3-(m+l)t^2-(1+k^2-ml)t+l.$$ 
By the Vieta formulas we have that 
\begin{align}\label{vieta 1}
	&	m+l=[X+\psi_2(X)+\psi_2^2(X):\mathbf 1]=2\epsilon_1+\alpha\epsilon_1+\beta\epsilon_2.
\end{align}
Repeating this for $Y$ we get
\begin{align}\label{vieta 2}
	k+n=[Y+\psi_2(Y)+\psi_2^2(Y):\mathbf 1]=2\epsilon_2+\gamma\epsilon_1 +\delta\epsilon_2 .
\end{align}
On the other hand, 
\begin{align*}
&1+m\epsilon_1+k\epsilon_2=[\psi_2(X^2):\mathbf 1]=[\psi_2(X)^2:\mathbf 1]=1+\alpha^2+\beta^2, \ \ \text{and}\\	
&1+l\epsilon_1+n\epsilon_2=[\psi_2(Y^2):\mathbf 1]=[\psi_2(Y)^2:\mathbf 1]=1+\gamma^2+\delta^2,
\end{align*}
and thus 
\begin{align*}
	\alpha^2+\beta^2+\gamma^2+\delta^2=(m+l)\epsilon_1+(k+n)\epsilon_2.
\end{align*}
Combining this together with Equations \eqref{vieta 1} and \eqref{vieta 2} we get
\begin{align*}
	\alpha^2+\beta^2+\gamma^2+\delta^2&=(2\epsilon_1+\alpha\epsilon_1+\beta\epsilon_2)\epsilon_1+(2\epsilon_2 +\gamma\epsilon_1+\delta\epsilon_2)\epsilon_2\\
	&=2+\alpha+\beta\epsilon_1\epsilon_2+2+\delta+\gamma\epsilon_1\epsilon_2,
\end{align*}
and so 
\begin{align}\label{eq: sum = 4}
	\alpha^2-\alpha +\beta^2-\beta \epsilon_1\epsilon_2+\gamma^2-\gamma\epsilon_1\epsilon_2+\delta^2-\delta=4.
\end{align}
Hence $|\alpha|, |\beta|,|\gamma|$ and $|\delta|$ are at most 2, which implies by \eqref{vieta 1} and $\eqref{vieta 2}$ that $m+l, k+n\leq 6$.

From each of the equalities $\psi_2(X^2)=\psi_2(X)^2$, $\psi_2(Y^2)=\psi_2(Y)^2$ and $\psi_2(XY)=\psi_2(X)\psi_2(Y)$ we get three equations on the parameters $\alpha, \beta, \gamma, \delta, k, l, m, n, \epsilon_1, \epsilon_2$. The bound $k,n\leq 6$ allows us to verify in Sage that the only solutions to said equations fulfilling \eqref{rank 3} and \eqref{eq: sum = 4}  are
\begin{align*}
	&k=1, l=1, m = 1, n = 0, \alpha =-1, \beta= 1, \gamma = -1, \delta = 0, \ \text{and}\\
	&k=1, l=1, m = 0, n = 1, \alpha= 0, \beta= -1, \gamma = 1, \delta = -1.
\end{align*}
By symmetry, it is enough to consider the second case. We have multiplication rules
\begin{align*}
	&X^2 = \mathbf 1 + Y  &Y^2 = \mathbf 1 + X + Y, &&XY = Y X = X +Y.
\end{align*}
Note that these are the same fusion rules as $\Ver_7^+$. Taking  dimension on the equalities above we arrive at the equation
 \begin{align}\label{eq: dim(Y)}
 	\dim(Y)^3-2\dim(Y)^2-\dim(Y)+1=0.
 \end{align}
On the other hand, note that $\mathcal C$ must be degenerate. In fact, if it was non-degenerate (see Section \ref{section:non-deg}) it would lift to a symmetric category over a field of characteristic zero 
\cite{E} and thus would have integer Frobenius-Perron dimensions \cite[ Theorem 9.9.26]{EGNO}. We compute
\begin{align*}
	0=\dim(\mathcal C)&= 1 + \dim(X)^2+\dim(Y)^2\\
	&=1 + 1 + \dim(Y)+\dim(Y)^2,
\end{align*}
and so 
\begin{align*}
	\dim(Y)^2=-2-\dim(Y).
\end{align*}
Replacing this into \eqref{eq: dim(Y)} we get 
 \begin{align*}
0	&=\dim(Y)^3-2\dim(Y)^2-\dim(Y)+1\\
&= \dim(Y)(-2-\dim(Y))-2(-2-\dim(Y))-\dim(Y)+1\\
&=-2\dim(Y)+2+\dim(Y)+4+2\dim(Y)-\dim(Y)+1\\
&=7.
\end{align*}
Thus we must have $p=7$. Since $\mathcal{K(C)}\cong \mathcal K(\Ver_7^+)$ we see that the fiber functor $F:\mathcal C \to \Ver_7$ gives an equivalence onto $\Ver_7^+$.
\end{proof}

\begin{remark}
	Theorem \ref{theo: rank 3} gives a positive answer for Question \ref{question 1} in the case $p=7$. 
\end{remark}

\section{Rank 4 symmetric fusion categories}\label{section: rank 4}

\subsection{Exactly two self-dual simple objects}
In this section we classify symmetric fusion categories $\mathcal C$ of rank $4$ with exactly two self-dual simple objects. Namely, we show that $\mathcal C$ is equivalent to one of the following:
\begin{itemize}
		\item If $p=2$, $\mathcal C\cong \operatorname{Vec}_{\mathbb Z_4}^{\mathbb Z_3}$ or $\mathcal C\cong \mathcal C\left(\mathbb Z_4,q\right)$, where $q:\mathbb Z_4\to \textbf k^{\times}$ is one of the two group maps satisfying $q(g)^2=1$ for all $g\in \mathbb Z_4$,  see \cite[Section 8.4]{EGNO}.
		\item If $p=3$,	 $\mathcal C\cong \mathcal C\left(\mathbb Z_4,q\right)$.
	\item  If $p>3$, $\mathcal C\cong \Rep(A_4)$ or $\mathcal C\cong \mathcal C\left(\mathbb Z_4,q\right)$.
\end{itemize}

We use the parametrization of based rings of rank 4 with exactly two self-dual basis elements as given in \cite{L}. Let $c, e, k, l, p, q$ be non-negative integers satisfying the following equations:
 \begin{align}\label{rank 4 1}
 	kl + lc = lp + kq,
 \end{align}
 \begin{align}\label{rank 4 2}
 	kp + le + kc = 2lq + k^2,
 \end{align}
\begin{align}\label{rank 4 3}
	l^2 + c^2 = 1 + q^2 + p^2,
\end{align}
 \begin{align}\label{rank 4 4}
 	l^2 + k^2 + q^2 = 1 + 2pk + qe.
 \end{align}
 We denote by $ K(c, e, k, l, p, q)$ the based ring with basis $1, X, Y, Z$ and multiplication given by
 \begin{align}\label{rank 4 rules}
 	\begin{aligned}
 	&X^2 = pX + lY + cZ && XY = Y X = qX + kY + lZ\\
 	&Y^2 = \mathbf 1 + kX + eY + kZ   &&Y Z = ZY = lX + kY + qZ\\
 	&Z^2 = cX + lY + pZ    &&XZ = ZX = \mathbf 1 + pX + qY + pZ.
 	\end{aligned}
 \end{align}
Any based ring of rank 4 with exactly two self-dual basis elements is of the form $ K(c, e, k, l, p, q)$ for some non-negative integers $c, e, k, l, p, q$ satisfying \eqref{rank 4 1}-\eqref{rank 4 4}, see \cite{L}.

When $p>2$, recall that for the second Adams operation $\psi_2: K(c, e, k, l, p, q) \to K(c, e, k, l, p, q)$ we have that $\psi_2(W)\equiv W^2 \mod 2$ for any $W\in K(c,e,k,l,p,q)$. Hence
\begin{align}\label{psi in rank 4}
&	\psi_2(X)=\alpha_1 X + \alpha_2 Y + \alpha_3 Z, && \psi_2(Y)=\epsilon\textbf 1+ \beta_1 X + \beta_2 Y + \beta_3 Z, && \psi_2(Z)=\gamma_1 X+\gamma_2 Y+\gamma_3 Z,
\end{align}
for some $\epsilon=\pm 1$ and $\alpha_i, \beta_i, \gamma_i \in \mathbb Z$, $i=1,2,3.$ Recall that $\psi_2(W)\equiv W^2 \mod 2$ and thus 
\begin{align*}
&	\alpha_1 \equiv p \equiv \gamma_3 \mod 2  &&\alpha_2\equiv l \equiv \gamma_2 \mod 2 && \alpha_3 \equiv c \equiv \gamma_1 \mod 2 \\ &\beta_1 \equiv k \equiv \beta_3 \mod 2 &&\beta_2 \equiv e \mod 2.
\end{align*}
We will use the previous congruences repeatedly throughout this section.

We start by discarding some possibilities for the Adams operation in the lemmas below.
\begin{lemma}\label{not Z}
	Let $\mathcal C$ be a symmetric fusion category of rank $4$ with exactly two self-dual simple objects. Then $\text{Im}({\psi_2}) \ne\mathbb Z$ . 
\end{lemma}

\begin{proof}
	Suppose that $\text{Im}({\psi_2}) =\mathbb Z$. Then by Proposition \ref{image is Q} we have that $Y^2=\textbf 1$ and $[X Z:Y]=1$, which imply $k=e=0$ and $q=1$. But then from Equation \eqref{rank 4 4} we get that $l^2+1=1$, thus $l=0$. Thus Equation \eqref{rank 4 3} is $c^2=2+p^2$, which has no integer solutions. 
\end{proof}

\begin{lemma}\label{order 2}
	Let $\mathcal C$ be a symmetric fusion category of rank $4$ with exactly two self-dual simple objects. Then $\psi_2$ does not have order 2. 
\end{lemma}

\begin{proof}
	Suppose for the sake of contradiction that $\psi_2^2=\text{Id}$.  Then from $\psi_2^2(X)=X$ and Equation \eqref{psi in rank 4}, we get the equations 
	\begin{align*}
		&0=\alpha_2\epsilon, &&1=\alpha_1^2+\alpha_2\beta_1+\alpha_3\gamma_1, &&0=\alpha_1\alpha_3 + \alpha_2 \beta_3 + \alpha_3 \gamma_3.
	\end{align*}
	Thus $\alpha_2=0$ and
	\begin{align}\label{order 2 1}
			 &1=\alpha_1^2 +\alpha_3\gamma_1, &&0= \alpha_3(\alpha_1 + \gamma_3).
	\end{align}
Similarly, from $\psi_2^2(Z)= Z$ we get that $\gamma_2=0$ and 
\begin{align}\label{order 2 2}
	&1 = \gamma_1\alpha_3 +\gamma_3^2, &&0=\gamma_1(\alpha_1+\gamma_3).
\end{align} 
	We divide the rest of the proof in two cases. 
	
	$\blacktriangledown$ Case 1: $\alpha_1\ne - \gamma_3$. Then by Equations \eqref{order 2 1} and \eqref{order 2 2} we must have $\alpha_3=0=\gamma_1$ and $\alpha_1^2=1=\gamma_3^2$, so $\gamma_3=\alpha_1=\pm 1$.  Then
	\begin{align*}
		\epsilon l = [\psi_2(X^2) : \mathbf 1] = [\psi_2(X)^2:\mathbf 1]= \alpha_2^2+2\alpha_1 \alpha_3=0,
	\end{align*}
	which implies $l=0$. On the other hand,
	\begin{align*}
		1+q\epsilon= [\psi_2(XZ) : \mathbf 1] = [\psi_2(X)\psi_2(Z):\mathbf 1]= \alpha_1\gamma_3  =1,
	\end{align*}
	which implies $q=0$. But then by Equation \eqref{rank 4 3} we have $p=0$ which contradicts $p\equiv \alpha_1 \mod 2$.

	$\blacktriangledown$ Case 2: $\alpha_1= -\gamma_3$. Then
		\begin{align*}
	&	\epsilon l = [\psi_2(X^2) : \mathbf 1] = [\psi_2(X)^2:\mathbf 1]= 2\alpha_1 \alpha_3=-2\gamma_3\alpha_3,\\
	&	\epsilon l = [\psi_2(Z^2) : \mathbf 1] = [\psi_2(Z)^2:\mathbf 1]= 2\gamma_1 \gamma_3,
	\end{align*}
and so $\gamma_1\gamma_3=-\gamma_3\alpha_3$. Suppose first that $\gamma_3\ne 0$. Then $\gamma_1=-\alpha_3$ and so by  Equation \eqref{order 2 1}
\begin{align*}
	1=\alpha_1^2 -\alpha_3^2,
\end{align*}
 which implies $\alpha_1=\pm 1$ and $\alpha_3=0=\gamma_1$. But then 
	\begin{align*}
	\epsilon l = [\psi_2(X^2) : \mathbf 1] = [\psi_2(X)^2:\mathbf 1]=0,
\end{align*}
so $l=0$ and 
	\begin{align*}
	p\alpha_1 = [\psi_2(X^2) : X] = [\psi_2(X)^2:X]= \alpha_1^2p=p.
\end{align*}
Since $p\equiv \alpha_1\equiv 1 \mod 2$ this implies $\alpha_1=1$, and thus $\gamma_3=-1$. But then
	\begin{align*}
		&[\psi_2(Z^2) : Z] =[\psi_2(cX+pZ):Z]=c\alpha_3+p\gamma_3=-p,\\
&[\psi_2(Z)^2:Z]= [(\gamma_3 Z)^2:Z]=\gamma_3^2 p=p,
\end{align*}
since $l,\alpha_3, \gamma_1$ and $\gamma_2$ are all 0 and $\gamma_3=-1$. Since the two equations above should be equal, then $p=0$ which contradicts $p\equiv 1 \mod 2$. 

The contradiction came from assuming $\gamma_3\ne 0$, and thus we should have $\gamma_3=0=\alpha_1$. Hence from the equations
\begin{align*}
	&[\psi_2(X^2) : \mathbf 1] =[\psi_2(pX+lY+cZ):\mathbf 1]=\epsilon l,\\
	&[\psi_2(X)^2: \mathbf 1] = [(\alpha_3Z)^2:\mathbf 1]=0
\end{align*}
 we get that $l=0$. On the other hand, by Equation \eqref{order 2 1} we have that $\alpha_3\gamma_1=1$, and so
\begin{align*}
&[\psi_2(XZ) : \mathbf 1] =[\psi_2(\textbf 1+pX+qY+pZ):\mathbf 1]=1+q\epsilon,\\
 &[\psi_2(X)\psi_2(Z):\mathbf 1]=[(\alpha_3Z)(\gamma_1X):\mathbf 1]=\alpha_3\gamma_1=1,
\end{align*}
 hence $q=0$. Then from Equations \eqref{rank 4 3} and \eqref{rank 4 4} we conclude $k=1=c$ and $p=0$. 	Moreover, $\psi_2^2(Y)=Y$ implies $0=\epsilon(\beta_2+1)$ and so $\beta_2=-1$. Using also that $\psi_2(XY)=\psi_2(X)\psi_2(Y)$ and $\psi_2(X)=\alpha_3(Z)$, we obtain
 	\begin{align*}
 	\epsilon &= [\psi_2(XY) : \mathbf 1] = [\psi_2(X)\psi_2(Y):\mathbf 1] =\alpha_3\beta_1,\\
 	\beta_1 &= [\psi_2(XY) : X] = [\psi_2(X)\psi_2(Y):X] =\alpha_3\beta_3,\\
 	-1&= [\psi_2(XY) : Y] = [\psi_2(X)\psi_2(Y):Y] =-\alpha_3\\
 	\beta_3 &= [\psi_2(XY) : Z] = [\psi_2(X)\psi_2(Y):Z] =\epsilon \alpha_3,
\end{align*}
 	which imply $1=\alpha_3=\gamma_1$ and $\epsilon=\beta_1=\beta_3$. Lastly
 	\begin{align*}
 		e	\epsilon &= [\psi_2(Y^2) : 1] = [\psi_2(Y)^2:1] =\beta_2^2+2\beta_1\beta_3=3.
 	\end{align*}

  Thus the fusion rules for $\mathcal C$ would be 
 \begin{align}\label{izumi-xu}
\begin{aligned}
	 	&X^2= Z &&XY=Y,\\
 	&Y^2=1 + X +3Y + Z &&ZY=Y,\\
 	&Z^2=X &&ZX=1.
\end{aligned}
 \end{align}
 These are the fusion rules of the Izumi-Xu category (see \cite{CMS}).
 The Frobenius Perron dimension of $\mathcal C$ is $\frac{21+2\sqrt{21}}{2}$, which is not possible in positive characteristic. In fact, dimensions in $\Ver_p$ are in the field $\mathbb Q(z+z^{-1})$, see \cite[Theorem 4.5 (iv)]{BEO}. Since we have a symmetric fiber functor $F:\mathcal C\to \Ver_p$, the same should be true for $\mathcal C$, which makes the obtained dimension  impossible. 
\end{proof}

\begin{remark}\label{Remark on Adams operation}
The Adams operation is not enough on its own to classify symmetric fusion categories in positive characteristic. In fact, in the proof of the previous Lemma we found a possible based ring \eqref{izumi-xu} and a suitable Adams operation, given by
\begin{align*}
	\psi_2(X)=Z, && \psi_2(Y)=\mathbf 1+X-Y+Z, &&\psi_2(Z)=X.
\end{align*}
However, as stated, there is no fusion category over a field of positive characteristic with \eqref{izumi-xu} as its Grothendieck ring.
\end{remark}

\begin{lemma}
	Let $\mathcal C$ be a symmetric fusion category of rank $4$ with exactly two self-dual simple objects. Then $\psi_2^2\ne \psi_2$. 
\end{lemma}

\begin{proof}
	Suppose that $\psi_2^2= \psi_2$. From the equality $\psi_2^2(X)= \psi_2(X)$ we get the equations 
	\begin{align*}
		&0=\alpha_2\epsilon, &&\alpha_1=\alpha_1^2+\alpha_2\beta_1+\alpha_3\gamma_1, &&\alpha_3=\alpha_1\alpha_3 + \alpha_2 \beta_3 + \alpha_3 \gamma_3.
	\end{align*}
Thus $\alpha_2=0$ and $\alpha_3= \alpha_3(\alpha_1 + \gamma_3)$. If $\alpha_3 \ne 0$ then $\alpha_1+\gamma_3 =1$  and so $p+p\equiv 1 \mod 2$ which is a contradiction. Thus we have $\alpha_3=0$ and $\alpha_1=1$ or $0$. Analogously, from $\psi_2^2(Z)= \psi_2(Z)$ we get $\gamma_2=0=\gamma_1$ and $\gamma_3=1$ or $0$. Lastly, from $\psi_2^2(Y)= \psi_2(Y)$ we have that $\beta_2=0$. 

On the other hand,
\begin{align*}
	\epsilon l = [\psi_2(X^2) : \mathbf 1] = [\psi_2(X)^2:\mathbf 1]= \alpha_2^2+2\alpha_1 \alpha_3=0,
\end{align*}
which implies $l=0$, and so
\begin{align*}
c\gamma_3 = [\psi_2(X^2):Z]  = [\psi_2(X)^2: Z] = c \alpha_1^2.	
\end{align*}
If $c=0$ then $0= 1+ q^2 + p^2$ by Equation \eqref{rank 4 3}, which is not possible. Thus $\gamma_3= \alpha_1^2= \alpha_1$. We divide the rest of the proof in two cases:

$\blacktriangledown$ Case 1: $\alpha_1=\gamma_3 =0$. Then 
\begin{align*}
	\epsilon k = [\psi_2(XY) :1] = [\psi_2(X) \psi_2(Y): 1] = \alpha_1\beta_3 + \alpha_3\beta_1 + \alpha_2\beta_2=0,
\end{align*}
and so $k=0$. Then by Equation \eqref{rank 4 4} we have $q^2= 1+qe$, thus $q=1$ and so $c^2=2+p^2$ by Equation \eqref{rank 4 3}, which has no integer solutions.

$\blacktriangledown$ Case 2: $\alpha_1=\gamma_3 =1$. 
Then
\begin{align*}
	&0 = [\psi_2(XY) :Y] = [\psi_2(X) \psi_2(Y): Y] =  q\beta_3,
\end{align*}
so $q=0$ or $0=\beta_3=\beta_1$. If $q=0$ then $c^2= 1 + p^2$ by Equation \eqref{rank 4 3} and so $c=1$ and  $p=0$. Then $1\equiv c \equiv \alpha_3 \equiv 0 \mod 2$, a contradiction. Thus we must have $0=\beta_3 $ and so 
\begin{align*}
	&\epsilon k = [\psi_2(XY) :\mathbf 1] = [\psi_2(X) \psi_2(Y): \mathbf 1] = \beta_3 =0,
\end{align*}
which implies $k=0$. Since  $q^2= 1+qe$ by Equation \eqref{rank 4 4}  we get $q=1$. But then  $c^2=2+p^2$ by Equation \eqref{rank 4 3}, which has no integer solutions. 
\end{proof}

We will need the following auxiliary lemma.
\begin{lemma}\label{not real K(C)}
	If $\mathcal C$ is a fusion category with commutative $\mathcal{K(C)}$ and a non-self-dual object, then there exists a ring homomorphism $\mathcal{K(C)}\to \mathbb C$ whose image is not contained in $\mathbb R$. 
\end{lemma}
\begin{proof}
	Let $X\in \mathcal C$ be a non-self-dual object, and consider the map of multiplication by $X-X^*$ in $\mathcal{K(C)}_{\mathbb C}$. In the basis given by simple objects,  we can represent this map by a non-trivial skew symmetric matrix. Thus its eigenvalues are zero or non-real. Since $X\not\cong X^*$ there must exist at least one non-real eigenvalue $\lambda\ne 0$. Hence $\mathcal{K(C)}_{\mathbb C}$ has a $1$ dimensional representation where $X$ acts as multiplication by $\lambda$.
\end{proof}

\begin{theorem}\label{rank 4 integral}
	Let $\mathcal C$ be a symmetric fusion category of rank 4 with exactly 2 self-dual simple objects. Then $\mathcal C$ is integral.
\end{theorem}

\begin{proof}
We proceed by looking at the different possibilities for $\mathcal{K(C)}_{\mathbb Q}$. Since $\mathcal{K(C)}_{\mathbb Q}$ is a semisimple commutative $\mathbb Q$-algebra of dimension $4$, we have five cases:

$\blacktriangledown$ Case 1:  $\mathcal{K(C)}_{\mathbb Q}\cong \mathbb Q \oplus \mathbb Q \oplus \mathbb Q \oplus \mathbb Q$.  Note that $\mathbb Q$-algebra maps $\mathcal{K(C)}_{\mathbb Q} \to \mathbb C$ are projections to $\mathbb Q$, and so this case is not possible by Lemma \ref{not real K(C)}.

$\blacktriangledown$ Case 2:  $\mathcal{K(C)}_{\mathbb Q}\cong \mathbb Q(\sqrt{n}) \oplus \mathbb Q \oplus\mathbb Q$. By Lemma \ref{not real K(C)} we must have $n<0$. But then $\FPdim_{\mathbb Q}: \mathcal{K(C)}_{\mathbb Q}\to \mathbb R$ can only have rational image and so $\mathcal C$ is integral.

$\blacktriangledown$ Case 3:  $\mathcal{K(C)}_{\mathbb Q}\cong \mathbb Q(\sqrt{n}) \oplus \mathbb Q(\sqrt{m})$ with $\mathbb Q(\sqrt n)\not\cong \mathbb Q(\sqrt m)$. Endomorphisms of $\mathbb Q(\sqrt{n}) \oplus \mathbb Q(\sqrt{m})$ are given by
\begin{align*}
	(1,0)\mapsto (1,0) &&(0,1)\mapsto (0,1) &&(\sqrt{n},0)\mapsto (\pm\sqrt{n},0) &&(0,\sqrt{m})\mapsto (0,\pm \sqrt{m}).
\end{align*}
These are all automorphisms of order 1 or 2, which is not possible for $(\psi_2)_{\mathbb Q}$ by Lemma \ref{order 2}.
Hence this case is discarded.

$\blacktriangledown$ Case 3:  $\mathcal{K(C)}_{\mathbb Q}\cong \mathbb Q(\sqrt{n}) \oplus \mathbb Q(\sqrt{n})$. By Lemma \ref{not real K(C)} we have $n<0$. But $\mathbb Q$-algebra morphisms from  $\mathbb Q(\sqrt{n}) \oplus \mathbb Q(\sqrt{n}) $ to $\mathbb C$  are embeddings onto $\mathbb Q(\sqrt{n})$. This contradicts the fact that $\FPdim : \mathbb Q(\sqrt{n}) \oplus \mathbb Q(\sqrt{n}) \to \mathbb C$ should have real image, and so this case is discarded.

$\blacktriangledown$ Case 4:  $\mathcal{K(C)}_{\mathbb Q}\cong \textbf F  \oplus \mathbb Q $, where $\textbf F$ is a field extension of degree 3 over $\mathbb Q$.  The only endomorphism of $\textbf F \oplus \mathbb Q$ with non-trivial kernel is given by $(a,b)\mapsto (b,b)$ for all $a\in \textbf F$ and $b\in \mathbb Q$. 
By Lemma \ref{not Z} we know that $(\psi_2)_{\mathbb Q}$ is not of this form.

On the other hand, non-trivial automorphisms of $\textbf F \oplus\mathbb Q$ have order 3 which is not possible for $(\psi_2)_{\mathbb Q}$ by Theorem \ref{even order}. Hence this case is also discarded.

$\blacktriangledown$ Case 5:  $\mathcal{K(C)}_{\mathbb Q}$ is a field extension of degree $4$ over $\mathbb Q$. Since $(\psi_2)_{\mathbb Q}\ne \text{Id}$ by Theorem \ref{psi not identity}, then it must be an automorphisms of order 2 or  4. Using Lemma \ref{order 2} we can discard the former possibility.  If $(\psi_2)_{\mathbb Q}$ has order $4$ then $Y, \psi_2(Y), \psi_2^2(Y)$ and $\psi_2^3(Y)$ are distinct roots of the minimal polynomial of $Y$. Since the minimal polynomial of $Y$ is given by $$m_Y(t)=t^4+(-2q-e)t^3+(2qe+q^2-k^2-l^2-1)t^2+(-q^2e+qk^2-lk^2+l^2e+2q)t+l^2-q^2,$$ then by the Vieta formulas the sum of the roots equals $2q+e$. Hence
\begin{align*}
	[Y+ \psi_2(Y)+\psi_2^2(Y)+\psi_2^3(Y):\mathbf 1]=2q+e.
\end{align*}
We compute 
\begin{align*}
	[\psi_2(Y):\mathbf 1]=\epsilon, &&[\psi_2^2(Y):\mathbf 1]=\epsilon+\beta_2\epsilon, &&[\psi_2^3(Y):\mathbf 1]=\epsilon+\beta_2\epsilon +\epsilon (\beta_1\alpha_2+\beta_2^2+\beta_3\gamma_2),
\end{align*}
and thus 
\begin{align*}
\epsilon + \epsilon+\beta_2\epsilon + \epsilon+\beta_2\epsilon +\epsilon (\beta_1\alpha_2+\beta_2^2+\beta_3\gamma_2)=	2q+e.
\end{align*}
Taking congruence mod 2 on both sides we get
\begin{align*}
	&e\equiv 1+1+e +1+e + kl+e+kl \equiv 1+e\mod 2, 
\end{align*}
which is not possible, so this case is also discarded.
\end{proof}

\begin{theorem}\label{rank 4}
Let $\mathcal C$ be an integral symmetric fusion category of rank 4 with exactly 2 self-dual simple objects. Then either
\begin{itemize}
	\item $\mathcal C\cong \mathcal C\left(\mathbb Z_4,q\right)$ and $p> 2$, or
	\item  $\mathcal C\cong \Rep(A_4)$ and $p>3$, or
	\item $\mathcal C\cong \operatorname{Vec}_{\mathbb Z_4}^{\mathbb Z_3}$ or $\mathcal C\cong \mathcal C\left(\mathbb Z_4,q\right)$ and $p=2$.
\end{itemize}
\end{theorem}
\begin{proof}
	
Let $\mathcal C$ be as in the statement with Grothendieck ring $K(c,e,k,l,p,q)$, see Equation \ref{rank 4 rules}. Since $X^*=Z$ we have that $\FPdim(X)=\FPdim(Z)$. Thus taking Frobenius-Perron dimensions on both sides of the fusion rule $Y^2=\mathbf 1+kX+eY+kZ$ we get 
\begin{align*}
	\FPdim(Y)(\FPdim(Y)-e)=1+2k\FPdim(X),
\end{align*} 	
	and so $\gcd(\FPdim(X),\FPdim(Y))=1$. From the fusion rule $XY= qX+kY+lZ$ we get $1=\frac{q+l}{\FPdim(Y)}+\frac{k}{\FPdim(X)}$ and so since the denominators are coprime either $\FPdim(Y)=q+l$  and $k=0$ or $q+l=0$ and $\FPdim(X)=k$. We split the rest of the proof in two cases.

$\blacktriangledown$ Case 1: $\FPdim(Y)=q+l$ and $k=0$. Since $Y^2=\mathbf 1+eY$ then $$\FPdim(Y)=\frac{e+\sqrt{e^2+4}}{2}.$$ But $\FPdim(Y)$ is an integer and then $e^2+4$ must be a square, so $e=0$. Hence $Y^2=\mathbf 1$ which implies $q+l=\FPdim(Y)=1$. Recall that $q$ and $l$ are non-negative integers, and so there are only two options: either $q=1$ and $l=0$, or $q=0$ and $l=1$. 	The former is not possible, since in that case $c^2=2+p^2$ by Equation \eqref{rank 4 3}, which has no integer solutions. Hence we have that $q=0$ and $l=1$. Moreover, the fusion rule $X^2=pX+Y+pZ$ implies that $$\FPdim(X)^2 = 2p\FPdim(X)+1.$$ 	This has  integer solutions only for $p=0$, in which case $\FPdim(X)=1$. Lastly by Equation \eqref{rank 4 3} we have $c=0$. Thus the fusion rules are
 \begin{align*}
	&X^2 = Y  && XY = Y X = Z\\
	&Y^2 = \mathbf 1    &&Y Z = ZY = X \\
	&Z^2 = Y   &&XZ = ZX = \mathbf 1.
\end{align*}
	Hence the category is pointed and  $\mathcal C\cong \mathcal C\left(\mathbb Z_4,q\right)$, where $q:\mathbb Z_4\to \textbf k^{\times}$ is a quadratic form satisfying   
	\begin{align*}
		q(gh)=q(g)q(h) b(g,h), 
	\end{align*}
where $b(g,h)=c_{Y,X} c_{X,Y} \in \Aut_{\mathcal C} (X,Y) \cong \mathbf{k}^\times$ for $X$ and $Y$ simple objects representing $g$ and $h$, respectively, see \cite[Lemma 8.4.2]{EGNO}. Since $\mathcal C$ is symmetric, it follows that $q$ is a group homomorphism. Finally,  from the fusion rules above we get $q(g)^2=1$ for all $g\in \mathbb Z_4$.
	
$\blacktriangledown$ Case 2: $\FPdim(X)=k$ and $q+l=0$. Since $q$ and $l$ are non-negative integers this implies $q=0=l$. The fusion rule $XZ=\mathbf 1+pX+pZ$ implies that $\FPdim(X)^2=1+2p\FPdim(X)$ and so $\FPdim(X)=1$ and $p=0$. On the other hand, from Equations \eqref{rank 4 3} and \eqref{rank 4 4} we get that $k=1=c$. Lastly, since $Y^2=\mathbf 1+X+eY+Z$ we have that $$\FPdim(Y)=\frac{e+\sqrt{e^2+12}}{2}.$$ Thus for $\FPdim(Y)$ to be an integer we need $e=2$. Consequently. the fusion rules are
  \begin{align*}
 	&X^2 = Z  && XY = Y X = Y\\
 	&Y^2 = \mathbf 1 +X +2Y +Z    &&Y Z = ZY = Y \\
 	&Z^2 = X   &&XZ = ZX = \mathbf 1.
 \end{align*}
	
Suppose $p>3$. The equations above imply that $\dim(X)=1=\dim(Z)$ and $\dim(Y)=-1$ or $3$. Hence $\dim(\mathcal C)\ne 0$ and thus we can lift $\mathcal C$ to a symmetric fusion category $\widetilde{\mathcal C}$ over a field $\textbf{f}$ in characteristic zero \cite[Section 4.1]{E}, which has the same Grothendick ring as $\mathcal C$. 
Thus $\widetilde{\mathcal{C}}$ is equivalent to $\Rep_{\textbf f}(A_4)$ \cite[Section 8.19]{D}, and so by uniqueness of the lifting we get that $\mathcal C$ is equivalent to $\Rep_{\textbf{k}}(A_4)$. 

	For the case $p=2$ we have that the objects $\textbf 1, X$ and $Z$ in $\mathcal C$ generate  a copy of $\Rep(\mathbb Z_3)$. Doing de-equivariantization by $\mathbb Z_3$ we obtain a symmetric fusion category $\mathcal C^{\mathbb Z_3}$ of dimension $4$ \cite[Section 2.7]{EGNO}. Hence $\mathcal C^{\mathbb Z_3}$ is  pointed and thus equivalent to $\operatorname{Vec}_{\mathbb Z_4}$. There is only one action of $\mathbb Z_3$ on $\mathbb Z_4$, and so doing equivariantization by $\mathbb Z_3$ gives us back $\mathcal C$. Hence
	$\mathcal C\cong \text{Vec}_{\mathbb Z_4}^{\mathbb Z_3}$.
	
	Lastly, for $p=3$ there is no category realizing these fusion rules. In fact, we know all symmetric fusion categories in characteristic $3$. Any such category is an equivariantization of a pointed category associated with a $3$-group by the action of a group $G$ of order relatively prime to $3$, see \cite[Section 8]{EOV}. Note that the group $G$ is non-trivial, since the category with the fusion rules above is not pointed as $\FPdim(Y)=3$. Thus the category should contain a non-trivial Tannakian subcategory $\Rep(G)$ of rank prime to $3$, which is not possible with the fusion rules above. 
\end{proof}

\begin{remark}
	In the proof of Theorem \ref{rank 4 integral} we showed that, when $\mathcal C$ is a symmetric fusion category of rank $4$ and exactly 2 self-dual simple objects, then the only possibilities for $\mathcal{K(C)}_{\mathbb Q}$ are $\mathbb Q^{\oplus 4}$ and $\mathbb Q(\sqrt{n})\oplus \mathbb Q^{\oplus 2}$, for $n$ a negative square-free integer. Moreover, Theorem \ref{rank 4} shows that such a category is equivalent to either $\mathcal C(\mathbb Z_4, q)$, $\Rep(A_4)$ or $\operatorname{Vec}_{\mathbb Z_4}^{\mathbb Z_3},$ the first of which satisfies $\mathcal{K(C)}_{\mathbb Q} \simeq \mathbb Q(\sqrt{-1})\oplus \mathbb Q^{\oplus 2}$, and the other two  $\mathcal{K(C)}_{\mathbb Q} \simeq \mathbb Q(\sqrt{-3})\oplus \mathbb Q^{\oplus 2}$.
\end{remark}

\subsection{All simple objects are self-dual}
We are not able to provide a classification of symmetric fusion categories of rank $4$, but here are some comments for the remaining case, in which all simple objects are self-dual. It follows from Theorem \ref{rango p-1/2} that we can have examples of such categories that are non super-Tannakian only  in characteristics $p=5$ or $7$.

We take a look first at the case $p=5$. 
\begin{proposition}
	Let $\mathcal C$ be a non-super-Tannakian symmetric fusion category of rank $4$ in characteristic $p=5$. Then $\mathcal K(\mathcal C)_{\mathbb Q}\cong \mathbb Q(\sqrt{5})\oplus\mathbb Q^{\oplus 2} $ or $\mathcal K(\mathcal C)_{\mathbb Q}\cong \mathbb Q(\sqrt{5})\oplus \mathbb Q(\sqrt{m})$ for some $m\in \mathbb Z$.
\end{proposition}

\begin{proof}
Since $\mathcal{K(C)}_{\mathbb Q}$ is a semisimple commutative $\mathbb Q$-algebra of dimension $4$, it can either be $\mathbb Q^{\oplus 4}$, $\mathbb Q^{\oplus 2}\oplus \mathbb Q(\sqrt{n})$, $\mathbb Q(\sqrt{n})\oplus \mathbb Q(\sqrt{m})$, $\mathbb Q(a) \oplus \mathbb Q$ or $\mathbb Q(b)$, for $n,m \in \mathbb Z$, and $a,b\in \mathbb Q$ such that $[\mathbb Q(a):\mathbb Q]=3$ and $[\mathbb Q(b):\mathbb Q]=4$.

Consider the Verlinde fiber functor $F:\mathcal C \to \Ver_5$, and let $\tilde F:\mathcal C \to \Ver_p^+$ be as in \eqref{tilde F}. We denote also by $\tilde F$ the induced $\mathbb Q$-algebra homomorphism $\mathcal{K(C)}_{\mathbb Q}\to 	\mathcal K(\Ver_p^+)_{\mathbb Q}$.  
 By the proof of Theorem \ref{rango p-1/2}, since $\mathcal C$ is not super-Tannakian then this map is surjective and so $$\tilde F: \mathcal{K(C)}_{\mathbb Q}\twoheadrightarrow\mathcal K(\Ver_5^+)_{\mathbb Q}\cong \mathbb Q(\xi_5+\xi_5^{-1})=\mathbb Q(\sqrt{5}).$$
 That is, the image of  $\mathcal{K(C)}_{\mathbb Q}$  under $\tilde F$ is $\mathbb Q(\sqrt{5})$. Then the only remaining possibilities for $\mathcal{K(C)}_{\mathbb Q}$ are $\mathbb Q^{\oplus 2}\oplus \mathbb Q(\sqrt{5})$ or $\mathbb Q(\sqrt{5})\oplus \mathbb Q(\sqrt{m})$ for some $m\in \mathbb Z$. 
%
\end{proof}

For $p=7$, we have the following result.

\begin{proposition}
	Let $\mathcal C$ be a non-super-Tannakian symmetric fusion category of rank $4$ in characteristic $p=7$. Then $\mathcal{K(C)}_{\mathbb Q}\cong \mathbb Q \oplus \mathbb Q(a)$ for some $a$ such that  $[\mathbb Q(a):\mathbb Q]=3$. Moreover, $\psi_2^3=\text{Id}$.
\end{proposition}
\begin{proof}
	 Since $\mathcal{K(C)}_{\mathbb Q}$ is a semisimple commutative algebra of dimension $4$, it can either be $\mathbb Q \oplus \mathbb Q \oplus \mathbb Q \oplus \mathbb Q$, $\mathbb Q(\sqrt{n})\oplus \mathbb Q \oplus \mathbb Q$, $\mathbb Q(\sqrt{n})\oplus \mathbb Q(\sqrt{m})$, $\mathbb Q(a)\oplus \mathbb Q$ or $\mathbb Q(b)$, for $a,b\in \mathbb C$ such that $[\mathbb Q(a):\mathbb Q]=3$ and $[\mathbb Q(b):\mathbb Q]=4.$ 
	 
	 Thus we have that if $f\in \End(\mathcal{K(C)}_{\mathbb Q})$ then either $f^n=\Id$ for $n=1, 2, 3, 4$, $f^k=f$ for $k=2, 3$ or $f^3=f^2$. By Theorem \ref{remark on powers of psi}, the only possibility for $(\psi_2)_{\mathbb Q} \in \End(\mathcal{K(C)}_{\mathbb Q})$ is that $(\psi_2)_{\mathbb Q}^3=\text{Id}$, which can only happen if $\mathcal{K(C)}\cong \mathbb Q \oplus \mathbb Q(a)$, as desired. 
 \end{proof}


\bibliographystyle{plain}

\end{document}